\documentclass[11pt]{amsart}
\usepackage{amsfonts}

\newtheorem{thm}{Theorem}[section]
\newtheorem{lem}[thm]{Lemma}
\newtheorem{cor}[thm]{Corollary}
\newtheorem{prop}[thm]{Proposition}

\newtheorem{ques}[thm]{Question}
\theoremstyle{definition}
\newtheorem{exmp}[thm]{Example}
\newtheorem{defn}[thm]{Definition}
\newtheorem{deflem}[thm]{Definition-Lemma}
\newtheorem{defprop}[thm]{Definition-Proposition}

\theoremstyle{remark}
\newtheorem{rem}[thm]{Remark}

\newcommand{\Z}{{\mathbb{Z}}}

\newcommand{\Ext}{\operatorname{Ext}}
\newcommand{\Hom}{\operatorname{Hom}}
\newcommand{\Endhom}{\operatorname{End}}

\newcommand{\add}{\mathrm{add}\; }

\newcommand{\ind}{\mathrm{ind}\text{-} }
\newcommand{\modu}{\mathrm{mod}\text{-} }

\newcommand{\tiltq}{\overrightarrow{\mathcal{T}}}
\newcommand{\ptiltq}{\overrightarrow{\mathcal{T}_{\mathrm{p}}}}

\begin{document}

\title[Distributive lattices and the poset of pre-projective tilting modules]{Distributive lattices and the poset of pre-projective tilting modules}
\author{Ryoichi Kase}
\address{Department of Pure and Applied Mathematics
Graduate School of Information Science and Technology
,Osaka University, Toyonaka, Osaka 560-0043, Japan}
\email{r-kase@cr.math.sci.osaka-u.ac.jp}
\date{}

\maketitle
\footnote[0]{ 
2000 \textit{Mathematics Subject Classification}.
Primary 16G20; Secondary 16D80.
}
\footnote[0]{ 
\textit{Key words and phrases}. 
Tilting modules, representations of quivers, distributive lattices.
}
\begin{abstract}
D.Happel and L.Unger defined a partial order on the set of basic tilting modules. We study the poset of basic pre-projective tilting modules over path algebra of infinite type. We give an
equivalent condition for that this poset is a distributive lattice.
We also give an equivalent condition for that a distributive lattice is isomorphic to the poset of  basic pre-projective tilting modules over path algebra of infinite type.

\end{abstract}
\section*{Introduction}
Tilting theory first appeared in the article by Brenner and Butler\;\cite{BB}. In this article the notion of a tilting module for finite dimensional algebra was introduced. Tilting theory  now appear  in many areas of mathematics, for example algebraic geometry, theory of algebraic groups and algebraic topology. Let $T$ be a tilting module for finite dimensional algebra $A$ and let $B=\Endhom_{A}(T)$. Then Happel showed that the two bounded derived categories
$\mathcal{D}^{\mathrm{b}}(A)$ and $\mathcal{D}^{\mathrm{b}}(B)$
are equivalent as triangulated category.  Therefore classifying tilting modules is an important problem.
 
  Theory of tilting-mutation introduced by Riedtmann and Schofield is one of the approach to this problem. Riedtmann and Schofield defined the tilting quiver related with tilting-mutation.  Happel and Unger defined the partial order on the set of basic tilting modules  and showed that tilting quiver is coincided with Hasse quiver of this poset. These combinatorial structure are now studied by many authors. 
  
  \subsection*{notations} Let $Q$ be a finite connected quiver without loops or oriented cycles. We denote by $Q_{0}$\;(resp.\;$Q_{1}$) the set of vertices\;(resp.\;arrows) of $Q$. For any arrow $\alpha\in Q_{1}$ we denote by $s(\alpha)$ its starting point and denote by $t(\alpha)$ its target point\;(i.e.\;$\alpha$ is an arrow from $s(\alpha)$ to $t(\alpha)$).   Let $kQ$ be the path algebra of $Q$ over an algebraically closed field $k$. Denote by $\modu kQ$ the category of finite dimensional right $kQ$ modules and by $\ind kQ$ the category of
 indecomposable modules in $\modu kQ$. For any module $M\in \modu kQ$ we denote by $|M|$ the number of
 pairwise non isomorphic indecomposable direct summands of $M$.       For any paths $w:a_{0}\stackrel{\alpha_{1}}{\rightarrow}a_{1}\stackrel{\alpha_{2}}{\rightarrow}\cdots \stackrel{\alpha_{r}}{\rightarrow} a_{r}$ and $w^{'}:b_{0}\stackrel{\beta_{1}}{\rightarrow}b_{1}\stackrel{\beta_{2}}{\rightarrow}\cdots \stackrel{\beta_{s}}{\rightarrow} b_{s}$, \[w\cdot w^{'}:=\left\{\begin{array}{ll}
 a_{0}\stackrel{\alpha_{1}}{\rightarrow}a_{1}\stackrel{\alpha_{2}}{\rightarrow}\cdots \stackrel{\alpha_{r}}{\rightarrow} a_{r}=b_{0}\stackrel{\beta_{1}}{\rightarrow}b_{1}\stackrel{\beta_{2}}{\rightarrow}\cdots \stackrel{\beta_{s}}{\rightarrow} b_{s} & \mathrm{if\ }a_{r}=b_{0} \\ 
 0 & \mathrm{if\ }a_{r}\neq b_{0},
 \end{array}\right.\]
 in $kQ$.  Let $P(i)$ be an indecomposable projective module in $\modu kQ$  associated with vertex $i\in Q_{0}$. 
 
 In this paper we will consider the set $\mathcal{T}_{\mathrm{p}}(Q)$ of basic pre-projective tilting modules  and study its combinatorial structure. In\;\cite{K}  we showed following:
 
 \begin{thm}
 If $Q$ satisfies the following condition $(\mathrm{C})$,
 \[(\mathrm{C})\ \  \delta(a):=\#\{\alpha\in Q_{1}\mid s(\alpha)=a\ \mathrm{or\ }t(\alpha)=a\}\geq 2\ \forall a\in Q_{0},\]
 then for any $T\in \mathcal{T}_{\mathrm{p}}$ there exists $(r_{i})_{i\in Q_{0}}\in \mathbb{Z}_{\geq 0}^{Q_{0}}$ such that $T\simeq \oplus_{i\in Q_{0}}\tau_{Q}^{-r_{i}}P(i).$

 Moreover $\oplus_{i\in Q_{0}}\tau^{-r_{i}}P(i)\mapsto (r_{i})_{i\in Q_{0}}$ induces a poset inclusion,
 \[(\mathcal{T}_{\mathrm{p}}(Q),\leq) \rightarrow (\mathbb{Z}^{Q_{0}},\leq^{\mathrm{op}}), \]
 where $(r_{i})\leq^{\mathrm{op}}(s_{i})\stackrel{\mathrm{def}}{\Leftrightarrow} r_{i}\geq s_{i}$ for any $i\in Q_{0}$.

 \end{thm}
 
  One of the result of this paper is that $Q$ satisfies the condition  $(\mathrm{C})$ if and only if $(\mathcal{T}_{\mathrm{p}}(Q),\leq)$ is a distributive lattice. We note that under the condition $(\mathrm{C})$ the poset $(\mathcal{T}_{\mathrm{p}}(Q),\leq)$
  has  inner poset inclusion $\tau^{-1}_{Q}$.  
 \begin{ques}
 \label{q1} 
 Let $L$ be a distributive lattice equipped with inner poset inclusion
 $\tau^{-1}$. When $(L,\tau^{-1})\simeq (\mathcal{T}_{\mathrm{p}}(Q),\tau_{Q}^{-1})$ for some $Q$?
 \end{ques} 
 
 As the goal of this paper we will give an answer of this question.
 Moreover we will construct a quiver $Q$  satisfying $(L,\tau^{-1})\simeq (\mathcal{T}_{\mathrm{p}}(Q),\tau^{-1}_{Q})$. 

 We now give an outline of this paper.  
 
 In Section\;1 we recall definitions of tilting modules, tilting quivers, lattices and distributive lattices.
 
 In Section\;2 we define the pre-projective part of tilting quiver and recall results of \cite{K}. 
 
 In Section\;3 we first show that $Q$ satisfies the condition $(\mathrm{C})$ if and only if $\mathcal{T}_{\mathrm{p}}(Q)$ is an infinite distributive lattice.  Next we give an answer of Question\;\ref{q1}.
 
\section*{Acknowledgement}
The author would like to express his gratitude to Professor Susumu Ariki for his mathematical supports and warm encouragements.

\section{Preliminary}

\subsection{Tilting modules}
In this sub-section we will recall the definition of tilting modules and basic results for combinatorics of the set of tilting modules.

\begin{defn}
A module  $T\in \modu kQ$ is tilting module if,\\
$(1)$\;$\Ext^{1}_{kQ}(T,T)=0$,\\
$(2)$\;$|T|=\#Q_{0}.$ 
 
\end{defn}

\begin{rem}
In general, a module $T$ over a finite dimensional algebra $A$ is called a tilting module if (1)\;its projective dimension
is finite, (2)\;$\Ext_{A}^{i}(T,T)=0$ for any $i>0$ and (3)\;there is a exact sequence,
\[0\rightarrow A_{A}\rightarrow T_{0}\rightarrow T_{1}\rightarrow\cdots \rightarrow T_{r}\rightarrow 0,\]
with $T_{i}\in \add T$. If $A$ is hereditary, then it is well-known that this definition is equivalent to our definition. 
 
\end{rem}

We denote by $\mathcal{T}(Q)$ the set of (isomorphism classes of) basic tilting modules in $\modu kQ$.

\begin{defprop}\cite[Lemma\;2.1]{HU2} Let $T,T^{'}\in \mathcal{T}(Q)$. Then the following relation $\leq$ define  a partial order on $\mathcal{T}(Q)$,
\[T\geq T^{'}\stackrel{\mathrm{def}}{\Leftrightarrow}\Ext^{1}_{kQ}(T,T^{'})=0.\]
\end{defprop}

\begin{defn}
The tilting quiver $\tiltq (Q)$ is defined as follows:\\
(1)\; $\tiltq (Q)_{0}:=\mathcal{T}(Q)$,\\
(2)\;$T\rightarrow T^{'}$ in $\tiltq (Q)$ if  $T\simeq M\oplus X$
, $T^{'}\simeq M\oplus Y$ for some $X,Y\in \ind kQ$, $M\in \modu kQ$ and there is a non split exact sequence,
\[0\rightarrow X \rightarrow M^{'}\rightarrow Y\rightarrow 0, \]
with $M^{'}\in \add M$.

\end{defn}

\begin{thm}\cite[Theorem\;2.1]{HU1}
The tilting quiver $\tiltq (Q)$ is coincided with the $\mathrm{Hasse}$-quiver
of $(\mathcal{T}(Q),\leq)$.
\end{thm}

\begin{rem}
In this paper we define the Hasse-quiver $\overrightarrow{P}$ of  (finite or infinite)  poset $(P,\leq)$ as follows:\\
(1)\;$\overrightarrow{P}_{0}:=P$,\\
(2)\;$x\rightarrow y$ in $\overrightarrow{P}$ if $x>y$ and there is no $z\in P$ such that $x>z>y$.

\end{rem}

\subsection{Lattices and distributive lattices}
In this subsection we will recall definition of a lattice  and  a distributive lattice. 

\begin{defn}
A poset $(L,\leq)$ is a lattice if for any $x,y\in L$ there is the  minimum element of $\{z\in L\mid z\geq x,y\}$ and there is 
the  maximum element of $\{z\in L\mid z\leq x,y\}$. 

In this case we denote by $x\vee y$ the minimum element of $\{z\in L\mid z\geq x,y\}$ and denote by $x\wedge y$ the maximum element
of $\{z\in L\mid z\leq x,y \}$. 
\end{defn}
\begin{defn}
A lattice $L$ is a distributive lattice if $(x\vee y)\wedge z=(x\wedge z)\vee (y\wedge z)$ holds
for any $x,y,z \in L$.  
\end{defn}

\begin{rem}
It is well-known that $L$ is a distributive lattice if and only if
$(x\wedge y)\vee z=(x\vee z)\wedge (y\vee z)$ holds for any $x,y,z \in L$.

\end{rem}

In this paper we use the following notation.

\begin{defn}
Let $(L_{1},\leq_{1})$ and $(L_{2},\leq_{2} )$ are posets and $\phi:L_{1}\rightarrow L_{2}$ be an order preserving map. \\
(1)\;We call $\phi$ a poset inclusion if $\phi(x)\leq_{2}\phi(y)$ implies $x\leq_{1} y$.\\
(2)\;Assume that $L_{1}$ and $L_{2}$ are lattices. We call $\phi$
lattice inclusion if $\phi$ is a poset inclusion and $\phi(x\vee y)=\phi(x)\vee \phi(y)$,\;$\phi(x\wedge y)=\phi(x)\wedge\phi(y)$ holds for any $x,y\in L_{1}$.
\end{defn}

\begin{defn}
Let $L$ be a lattice. We call an element $x\in L$ join-irreducible if $x=y\vee z$ implies either $y=x$ or $z=x$.
\end{defn}

\begin{defn}
Let $P$ be a poset and $I\subset P$. We call $I$ poset-ideal of $P$ if $x\leq y\in I$ implies $x\in I$.

Then we denote by $\mathcal{I}(P)$ the poset $(\{I:\mathrm{poset\text{-}ideal\ of }\ P \},\subset)$ and call it the ideal-poset of $P$.
\end{defn}

\begin{thm}
\label{brt}
$($Birkhoff's representation theorem,\;$\mathrm{c.f.}$\;\cite{B},\;\cite{Gr}$)$
Let $L$ be a finite distributive lattice and $J\subset L$ be the poset of join-irreducible elements of $L$.
Then $L$ is isomorphic to $\mathcal{I}(J)$.

\end{thm}

\section{Pre-projective tilting modules }
In this section we will review \cite{K}. Denote by $\tau=\tau_{Q}$ the Auslander-Reiten translation of $kQ$. First we collect basic properties of the Auslander-Reiten translation.

\begin{prop}
\label{arthm}
$(\mathrm{cf}.$\cite{ARS},\;\cite{ASS},\;\cite{G}$)$
Let $A=kQ$ be a path algebra and $M,N\in \ind A$.
Then the following assertions hold.\\
$(1)$\;If $M$ and $N$ are non-injective modules, then
\[\Hom_{A}(M,N)\simeq \Hom_{A}(\tau^{-1}M,\tau^{-1}N). \]
$(2)$\;$(\mathrm{Auslander}$-$\mathrm{Reiten\ duality})$\;There is a functorial isomorphism,
\[D\Hom_{A}(M,N)\simeq \Ext_{A}^{1}(N,\tau M),\]
where $D:=\Hom_{k}(-,k)$.\\
$(3)$\;For any indecomposable non-projective module $X$ and almost
split sequence
\[0\rightarrow \tau X\rightarrow E\rightarrow X\rightarrow 0,\]
we get
\[\mathrm{dim}\Hom(M,\tau X)-\mathrm{dim}\Hom(M,E)+\mathrm{dim}\Hom(M,X)=\left\{\begin{array}{ll}
1 & X\simeq M \\ 
0 & \mathrm{otherwise.}
\end{array}
\right.\]

\end{prop}

\begin{defn}
Let $\ptiltq(Q)$ be a full sub-quiver of $\tiltq(Q)$ with $\ptiltq(Q)_{0}=\mathcal{T}_{\mathrm{p}}(Q)$.
\end{defn}

\begin{lem}
$\ptiltq(Q)$ is coincided with the $\mathrm{Hasse}$-quiver of $(\mathcal{T}_{\mathrm{p}}(Q),\leq)$. 
\end{lem}

Now we consider the condition,
 \[(\mathrm{C})\ \ \delta(x):=\#\{\alpha\in Q_{1}\mid s(\alpha)=x\ \mathrm{or\ } t(\alpha)=x\}\geq 2 \ \mathrm{for\ any\ }x\in Q_{0}.\]
 Let $d(X,Y):=\mathrm{dim}$\;$\Ext_{kQ}^{1}(X,Y)$. If $Q$ satisfies
 the condition (C), then $kQ$ is representation infinite and pre-projective part of its Auslander-Reiten quiver is the translation quiver $\Z_{\leq 0}Q$ (cf.\cite{ASS}).
 Let $Y=\tau^{-s}P(y)$. Then Proposition\;\ref{arthm} implies the following.
\[
d(\tau^{-r}P(x),Y)    =\left\{\begin{array}{ll}
0 & \mathrm{if\ }(r,x)\not\succeq (s,y)  \\ 
1 & \mathrm{if\ }(r,x)=(s,y) \\ 
\begin{array}{l}
  \sum_{\alpha:s(\alpha)=x}d(\tau^{-r}P(t(\alpha)),Y) \\ 
 + \sum_{\alpha:t(\alpha)=x}d(\tau^{-r+1}P(s(\alpha)),Y)  \\ 
 - d(\tau^{-r+1}P(x),Y) 
 \end{array}  &  \mathrm{if\ }(r,x)\succ (s,y), \\
\end{array} \right.
\]
where $(r,x)\succeq (s,y)$ means either (i)\:$r>s$ or (ii)\;$r=s$ and there is a path from $x$ to $y$ hold.

\begin{lem}
\label{l}
Assume $Q$ satisfies the condition $(\mathrm{C})$. If there is an arrow $\gamma : x\rightarrow y$ in $Q$, then 
\[\mathrm{dim\;}\Ext_{kQ}^{1}(\tau^{-r}P(y),M)\leq \mathrm{dim\;}\Ext_{kQ}^{1}(\tau^{-r}P(x),M)\leq \mathrm{dim\;}\Ext_{kQ}^{1}(\tau^{-r-1}P(y),M).\]
for any $r\geq 0$ and $M\in \modu kQ$.
\end{lem}

We define a map $\l_{Q}:Q_{0}\times Q_{0}\rightarrow \mathbb{Z}_{\geq 0}$ as follows: Let $\tilde{Q}$ be a quiver with $\tilde{Q}_{0}:=Q_{0}$ and
 $\tilde{Q}_{1}:=Q_{1}\coprod -Q_{1}$ where for any arrow $\alpha: x\rightarrow y$ in $Q$ we set $-\alpha:y\rightarrow x$. For any path $w:x_{0}\stackrel{\alpha_{1}}{\rightarrow} x_{1}\stackrel{\alpha_{2}}{\rightarrow}\cdots \stackrel{\alpha_{r}}{\rightarrow} x_{r}$ in $\tilde{Q}$ we put $c^{+}(w):=\#\{i\mid \alpha_{i}\in Q_{1}\}$. Then we set $l_{Q}(x,y):=\mathrm{min}\{c^{+}(w)\mid w:\mathrm{path\ from\ }x\ \mathrm{to}\ y\ \mathrm{in}\ \tilde{Q}\}$.

\begin{prop}
\label{p}
If $Q$ satisfies the condition $(\mathrm{C})$, 
 then
 \[\Ext_{kQ}^{1}(\tau^{-r}P(i),\tau^{-s}P(j))=0\Leftrightarrow r\leq s+l_{Q}(j,i)\]
 \end{prop}
 \begin{proof} 
 $(\Rightarrow)$: Let $w:j=x_{0}\stackrel{\alpha_{1}}{\rightarrow}x_{1}\rightarrow
 \cdots \stackrel{\alpha_{t}}{\rightarrow}x_{t}=i$ be a path such that
 $l(j,i):=l_{Q}(j,i)=c^{+}(w)$ and $\{k_{1}<k_{2}<\cdots<k_{l(j,i)}\}=
 \{k\mid \alpha_{k}\in Q_{1}\}.$ If there exists $r>l(j,i)$ such that
 $\Ext^{1}(\tau^{-r}P(i),P(j))=0$, then, by Lemma\;\ref{l}, we obtain
 \[\begin{array}{l}
 0=d(\tau^{-r}P(x_{t}),P(j))\geq d(\tau^{-r}P(x_{k_{l(j,i)}}),P(j))\geq d(\tau^{-r+1}P(x_{k_{l(j,i)}-1}),P(j))\\ \geq\cdots \geq d(\tau^{-r+l(j,i)-1}P(x_{k_{1}}),P(j)) 
 \geq d(\tau^{-r+l(j,i)}P(x_{k_{1}-1}),P(j))\\ \geq d(\tau^{-r+l(j,i)}P(j),P(j))
 \geq \cdots \geq d(\tau^{-1}P(j),P(j))>0.
 \end{array}\]
 Therefore we get a contradiction. In particular if $\Ext_{kQ}^{1}(\tau^{-r}P(i),\tau^{-s}P(j))=0$ with $r > s+l(j,i)$, then by Proposition\;\ref{arthm}, we obtain a contradiction.
 
$(\Leftarrow)$ Let $\mathcal{A}(j):=\{(i,r)\mid r\leq l(j,i),\ \Ext_{kQ}^{1}(\tau^{-r}P(i),P(j))\neq 0\}$. If $\mathcal{A}(j)\neq \emptyset$,
 then we take $r:=\mathrm{min}\{r\mid (i,r)\in \mathcal{A}(j)\ \mathrm{for\ some\ }i\}$. Let $i\in Q_{0}$ such that $(i,r)\in \mathcal{A}(j)$ and
$(i^{'},r)\notin \mathcal{A}(j)$ for any $i^{'}\leftarrow i$ in $Q$.
Since 
\[0< d(\tau^{-r}P(i),P(j))\leq \sum_{\alpha:s(\alpha)=i }d(\tau^{-r}P(t(\alpha)),P(j))+\sum_{\beta:t(\beta)=i}d(\tau^{-r+1}P(s(\beta)),P(j)),\]
we obtain (1)\;$d_{j}(t(\alpha)+rn)\neq 0 $ for some $\alpha \in s(i)$ or (2)\;$d_{j}(s(\beta)+(r-1)n)\neq 0$ for some $\beta \in t(i)$.
Note that $r\leq l(j,i)\leq l(j,t(\alpha))$ for any $\alpha\in Q_{1}$ with $s(\alpha)=i$ and $r-1\leq l(j,i)-1 \leq l(j,s(\beta))$ for any $\beta\in Q_{1}$ with $t(\beta)=i$.
By the definition of $(i,r)$, we obtain  $d(\tau^{-r}P(t(\alpha)),P(j))=0=d(\tau^{-r+1}P(s(\beta)),P(j))$ for any $\alpha\in Q_{1}$ with $s(\alpha)=i$ and $\beta \in Q_{1}$ with $t(\beta)=i$. Therefore  we get a contradiction. In particular we obtain 
$\mathcal{A}(j)=\emptyset.$ 

Suppose there exists $(i,r,s)\in Q_{0}\times \mathbb{Z}_{\geq 0}\times\mathbb{Z}_{\geq 0}$
 such that $r\leq s+l(j,i)$ and \[\Ext_{kQ}^{1}(\tau^{-r}P(i),\tau^{-s}P(j))\neq 0.\] 
 If $r< s$, then Proposition\;\ref{arthm} shows $\Ext_{kQ}^{1}(\tau^{-r}P(i),\tau^{-s}P(j))= 0$. If  $r\geq s$, then Proposition\;\ref{arthm} implies $(i,r-s)\in \mathcal{A}(j)$. Therefore we obtain a contradiction.    
   
 \end{proof}

We note that in the proof of $(\Leftarrow)$ of the above Proposition we did not use the condition $(\mathrm{C})$. In particular we obtain
the following Corollary.

\begin{cor}
\label{c1}
Let $i,j\in Q_{0}$ and $r,s\in \mathbb{Z}_{\geq 0}$. If $r\leq s+l_{Q}(i,j)$, then
\[\Ext_{kQ}^{1}(\tau^{-r}P(i),\tau^{-s}P(j))=0.\] 

\end{cor}

\begin{thm}
\label{t0}
Assume that $Q$ satisfies the condition $(\mathrm{C})$. Then we get followings:\\
$(1)$\;Let $T\in \mathcal{T}_{\mathrm{p}}(Q)$. Then there exists
$(r_{x})_{x\in Q_{0}}\in \mathbb{Z}^{Q_{0}}_{\geq 0} $ such that
$T=\bigoplus_{x\in Q_{0}}\tau^{-r_{x}}P(x)$.\\
$(2)$\; $\bigoplus_{x\in Q_{0}}\tau^{-r_{x}}P(x)\rightarrow (r_{x})_{x\in Q_{0}}$ induces both a poset inclusion, \[\mathcal{T}_{\mathrm{p}}(Q)\rightarrow (\mathbb{Z}^{Q_{0}}_{\geq 0},\leq^{\mathrm{op}})\] and a quiver inclusion, \[\ptiltq(Q)\rightarrow \overrightarrow{(\mathbb{Z}^{Q_{0}}_{\geq 0},\leq^{\mathrm{op}})}\].
In this case we set $T_{x}:=r_{x}$ for any $T\simeq \bigoplus_{x\in Q_{0}}\tau^{-r_{x}}P(x)$.

\end{thm}

\begin{rem}
Assume that $Q$ satisfies the condition $(\mathrm{C})$. We define \[T(a):=\bigoplus_{x\in Q_{0}}\tau^{-l_{Q}(a,x)}P(x)\] for any $a\in Q_{0}$. Then $\tau^{-r}T(a)$ is the minimum element of $\{\mathcal{T}_{\mathrm{p}}(Q)\ni T\simeq \bigoplus_{x\in Q_{0}}\tau^{-r_{x}}P(x)\mid r_{a}\leq r \}$.

\end{rem}
\begin{proof}
Proposition\;\ref{p} shows that $T(a)$ is a minimum element of
$\{T\in \mathcal{T}_{\mathrm{p}}(Q)\mid P(a)\in \add T\}$. Let
$T\simeq \bigoplus_{x\in Q_{0}}\tau^{-r_{x}}P(s) \in \mathcal{T}_{\mathrm{p}}(Q)$ such that $r_{a}\leq r$ and $T^{'}=\bigoplus_{x\in Q_{0}}\tau^{-r^{'}_{x}}P(x)$ where $r^{'}_{x}:=\mathrm{max}\{r_{x},r+l_{Q}(a,x)\}$ for any $x\in Q_{0}$. It is easy to check that $T^{'}\in \mathcal{T}_{\mathrm{p}}(Q)$. Since $r^{'}_{x}\geq r$ for any $x\in Q_{0}$, we have $\tau^{r}T^{'}$ is a basic pre-projective tilting module with $P(a)\in \add T^{'}$. In particular we obtain $\tau^{r}T^{'}\geq T(a)$. Therefore we have $T^{'}\geq \tau^{-r}T(a)$.  
\end{proof}

\section{Main results}

In this section we give our main results. Denote by $\mathcal{Q}$ the set of finite connected  quivers without loops or oriented cycles.  First we show the following Theorem.

\begin{thm}
\label{t1}
Let $Q\in \mathcal{Q}$. Then
$\mathcal{T}_{\mathrm{p}}(Q)$ is an infinite distributive lattice if and only if $Q$ satisfies the condition $(\mathrm{C})$. 

\end{thm}

\begin{proof}
First we assume that  $Q$ doesn't satisfy the condition $(\mathrm{C})$. Then one of the following holds,\\
(a)\;there is a source $s$ in $Q$ such that $\delta(s)=1$,\\
(b)\;there is a sink $s$ in $Q$ such that $\delta(s)=1$.

In the case (a), let $x$ be the unique direct successor of $s$. We
denote by $I$ the set of successors of $x$. Let $C:=(\oplus_{i\in I}\tau^{-2}P(i))\oplus (\oplus_{i\not\in I}\tau^{-1}P(i))$. Then we consider following five modules $T:=P(s)\oplus\tau^{-1}P(x)\oplus C$, $T^{'}:=\tau^{-2}P(s)\oplus \tau^{-2}P(x)\oplus C$, $X_{1}:=\tau^{-1}P(s)\oplus \tau^{-1}P(x)\oplus C$, $X_{2}:=\tau^{-1}P(s)\oplus \tau^{-2}P(x)\oplus C$ and $Y:=P(s)\oplus \tau^{-2}P(s)\oplus C$.  We note that Corollary\;\ref{c1} implies $T$, $T^{'}$, $X_{1}$ and $X_{2}$ are in $\mathcal{T}_{\mathrm{p}}$. We also note that $\Ext^{1}_{kQ}(\tau^{-2}P(s),P(s))=0$. Indeed \[\begin{array}{lll}
\mathrm{dim\;}\Ext^{1}_{kQ}(\tau^{-2}P(s),P(s)) & = & \mathrm{dim\;}\Hom_{kQ}(P(s),\tau^{-1}P(s)) \\ 
 & = & \mathrm{dim\;}\Hom_{kQ}(P(s),\tau^{-1}P(x))-\mathrm{dim\;}\Hom_{kQ}(P(s),P(s))\ \\ 
 & = & \mathrm{dim\;}\Hom_{kQ}(P(s),\bigoplus_{y\rightarrow x}P(y))\\
 & & +\mathrm{dim\;}\Hom_{kQ}(P(s),\bigoplus_{z\leftarrow x}\tau^{-1}P(z))\\
 & & -\mathrm{dim\;}\Hom_{kQ}(P(s),P(x))\\
 & = & \mathrm{dim\;}\Hom_{kQ}(P(s),\bigoplus_{y\rightarrow x}P(y))\\
 & &+ \mathrm{dim\;}\Ext^{1}_{kQ}(\bigoplus_{z\leftarrow x}\tau^{-2}P(z),P(s))\\
 & &-\mathrm{dim\;}\Hom_{kQ}(P(s),P(x))\\
 & = & 0.
\end{array} \]
 Therefore we obtain $Y\in \mathcal{T}_{\mathrm{p}}(Q)$. Since there is the following diagram in $\ptiltq(Q)$,
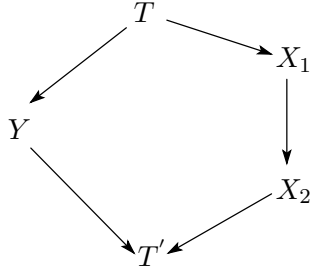
 \[
\unitlength 0.1in
\begin{picture}( 14.6500, 13.8000)( 20.7000,-20.1000)
\put(27.3000,-8.0000){\makebox(0,0)[lb]{$T$}}%
%
\special{pn 8}%
\special{pa 2690 830}%
\special{pa 2190 1220}%
\special{fp}%
\special{sh 1}%
\special{pa 2190 1220}%
\special{pa 2256 1196}%
\special{pa 2232 1188}%
\special{pa 2230 1164}%
\special{pa 2190 1220}%
\special{fp}%
\put(20.7000,-14.1000){\makebox(0,0)[lb]{$Y$}}%
%
\special{pn 8}%
\special{pa 2190 1460}%
\special{pa 2730 2010}%
\special{fp}%
\special{sh 1}%
\special{pa 2730 2010}%
\special{pa 2698 1948}%
\special{pa 2694 1972}%
\special{pa 2670 1976}%
\special{pa 2730 2010}%
\special{fp}%
\put(34.7000,-10.6000){\makebox(0,0)[lb]{$X_{1}$}}%
\put(34.7000,-17.5000){\makebox(0,0)[lb]{$X_{2}$}}%
\put(27.5000,-20.8000){\makebox(0,0)[lb]{$T^{'}$}}%
%
\special{pn 8}%
\special{pa 2902 790}%
\special{pa 3448 968}%
\special{fp}%
\special{sh 1}%
\special{pa 3448 968}%
\special{pa 3392 928}%
\special{pa 3398 952}%
\special{pa 3378 966}%
\special{pa 3448 968}%
\special{fp}%
%
\special{pn 8}%
\special{pa 3530 1100}%
\special{pa 3530 1540}%
\special{fp}%
\special{sh 1}%
\special{pa 3530 1540}%
\special{pa 3550 1474}%
\special{pa 3530 1488}%
\special{pa 3510 1474}%
\special{pa 3530 1540}%
\special{fp}%
%
\special{pn 8}%
\special{pa 3450 1700}%
\special{pa 2910 2010}%
\special{fp}%
\special{sh 1}%
\special{pa 2910 2010}%
\special{pa 2978 1994}%
\special{pa 2956 1984}%
\special{pa 2958 1960}%
\special{pa 2910 2010}%
\special{fp}%
\end{picture}%
\] 
  $\mathcal{T}_{\mathrm{p}} $ is not
 a distributive lattice.
 
 Similarly, in the case (b), we obtain that $\mathcal{T}_{\mathrm{p}}(Q) $ is not
 a distributive lattice.
 
 Next we assume  $Q$ satisfies the condition (C). Then Theorem\;\ref{t0} implies that $\mathcal{T}_{\mathrm{p}}$ is an infinite distributive lattice. Indeed it is easy to check that
for any basic pre-projective modules $T\simeq\bigoplus_{x\in Q_{0}}\tau^{-r_{x}}P(x)$ and $T^{'}\simeq \bigoplus_{x\in Q_{0}}\tau^{-r^{'}_{x}}P(x)$, both $\bigoplus_{x\in Q_{0}}\tau^{-\mathrm{min}\{r_{x},r^{'}_{x}\}}P(x)$ and $\bigoplus_{x\in Q_{0}}\tau^{-\mathrm{max}\{r_{x},r^{'}_{x}\}}P(x)$ are also basic pre-projective tilting modules\;(Remark.\;Let $a:=(r_{x})_{x\in Q_{0}},b=(r^{'}_{x})_{x\in Q_{0}}\in \mathbb{Z}_{\geq 0}^{Q_{0}}$. Then it is obvious that $a\vee b=(\mathrm{min}\{r_{x},r^{'}_{x}\})_{x\in Q_{0}}$ and $a\wedge b=(\mathrm{max}\{r_{x},r^{'}_{x}\})_{x\in Q_{0}}$ in the distributive lattice $(\mathbb{Z}^{Q_{0}},\leq^{\mathrm{op}})$.).

\end{proof}

\begin{exmp}
We give three examples of $\ptiltq(Q)$.

(1)\;Let $Q$ be the quiver:
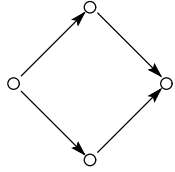
\[
\unitlength 0.1in
\begin{picture}(  8.6000,  8.6000)(  7.7000,-12.3000)
%
\special{pn 8}%
\special{ar 800 800 30 30  0.0000000 6.2831853}%
%
\special{pn 8}%
\special{ar 1200 400 30 30  0.0000000 6.2831853}%
%
\special{pn 8}%
\special{ar 1600 800 30 30  0.0000000 6.2831853}%
%
\special{pn 8}%
\special{ar 1200 1200 30 30  0.0000000 6.2831853}%
%
\special{pn 8}%
\special{pa 840 760}%
\special{pa 1170 430}%
\special{fp}%
\special{sh 1}%
\special{pa 1170 430}%
\special{pa 1110 464}%
\special{pa 1132 468}%
\special{pa 1138 492}%
\special{pa 1170 430}%
\special{fp}%
%
\special{pn 8}%
\special{pa 840 840}%
\special{pa 1170 1170}%
\special{fp}%
\special{sh 1}%
\special{pa 1170 1170}%
\special{pa 1138 1110}%
\special{pa 1132 1132}%
\special{pa 1110 1138}%
\special{pa 1170 1170}%
\special{fp}%
%
\special{pn 8}%
\special{pa 1240 1160}%
\special{pa 1570 830}%
\special{fp}%
\special{sh 1}%
\special{pa 1570 830}%
\special{pa 1510 864}%
\special{pa 1532 868}%
\special{pa 1538 892}%
\special{pa 1570 830}%
\special{fp}%
%
\special{pn 8}%
\special{pa 1240 430}%
\special{pa 1570 760}%
\special{fp}%
\special{sh 1}%
\special{pa 1570 760}%
\special{pa 1538 700}%
\special{pa 1532 722}%
\special{pa 1510 728}%
\special{pa 1570 760}%
\special{fp}%
\end{picture}%
\]

Then $\ptiltq(Q)$ is given by the following:
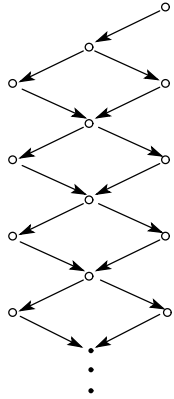
\[
\unitlength 0.1in
\begin{picture}(  8.5000, 20.3000)(  9.8000,-26.1000)
%
\special{pn 8}%
\special{ar 1800 600 20 20  0.0000000 6.2831853}%
%
\special{pn 8}%
\special{pa 1770 620}%
\special{pa 1440 780}%
\special{fp}%
\special{sh 1}%
\special{pa 1440 780}%
\special{pa 1510 770}%
\special{pa 1488 758}%
\special{pa 1492 734}%
\special{pa 1440 780}%
\special{fp}%
%
\special{pn 8}%
\special{ar 1400 810 20 20  0.0000000 6.2831853}%
%
\special{pn 8}%
\special{pa 1370 830}%
\special{pa 1040 990}%
\special{fp}%
\special{sh 1}%
\special{pa 1040 990}%
\special{pa 1110 980}%
\special{pa 1088 968}%
\special{pa 1092 944}%
\special{pa 1040 990}%
\special{fp}%
%
\special{pn 8}%
\special{ar 1000 1000 20 20  0.0000000 6.2831853}%
%
\special{pn 8}%
\special{ar 1800 1000 20 20  0.0000000 6.2831853}%
%
\special{pn 8}%
\special{pa 1770 1020}%
\special{pa 1440 1180}%
\special{fp}%
\special{sh 1}%
\special{pa 1440 1180}%
\special{pa 1510 1170}%
\special{pa 1488 1158}%
\special{pa 1492 1134}%
\special{pa 1440 1180}%
\special{fp}%
%
\special{pn 8}%
\special{ar 1400 1210 20 20  0.0000000 6.2831853}%
%
\special{pn 8}%
\special{pa 1370 1230}%
\special{pa 1040 1390}%
\special{fp}%
\special{sh 1}%
\special{pa 1040 1390}%
\special{pa 1110 1380}%
\special{pa 1088 1368}%
\special{pa 1092 1344}%
\special{pa 1040 1390}%
\special{fp}%
%
\special{pn 8}%
\special{ar 1000 1400 20 20  0.0000000 6.2831853}%
%
\special{pn 8}%
\special{ar 1800 1400 20 20  0.0000000 6.2831853}%
%
\special{pn 8}%
\special{pa 1770 1420}%
\special{pa 1440 1580}%
\special{fp}%
\special{sh 1}%
\special{pa 1440 1580}%
\special{pa 1510 1570}%
\special{pa 1488 1558}%
\special{pa 1492 1534}%
\special{pa 1440 1580}%
\special{fp}%
%
\special{pn 8}%
\special{ar 1400 1610 20 20  0.0000000 6.2831853}%
%
\special{pn 8}%
\special{pa 1370 1630}%
\special{pa 1040 1790}%
\special{fp}%
\special{sh 1}%
\special{pa 1040 1790}%
\special{pa 1110 1780}%
\special{pa 1088 1768}%
\special{pa 1092 1744}%
\special{pa 1040 1790}%
\special{fp}%
%
\special{pn 8}%
\special{ar 1000 1800 20 20  0.0000000 6.2831853}%
%
\special{pn 8}%
\special{ar 1800 1800 20 20  0.0000000 6.2831853}%
%
\special{pn 8}%
\special{pa 1770 1820}%
\special{pa 1440 1980}%
\special{fp}%
\special{sh 1}%
\special{pa 1440 1980}%
\special{pa 1510 1970}%
\special{pa 1488 1958}%
\special{pa 1492 1934}%
\special{pa 1440 1980}%
\special{fp}%
%
\special{pn 8}%
\special{ar 1400 2010 20 20  0.0000000 6.2831853}%
%
\special{pn 8}%
\special{pa 1370 2030}%
\special{pa 1040 2190}%
\special{fp}%
\special{sh 1}%
\special{pa 1040 2190}%
\special{pa 1110 2180}%
\special{pa 1088 2168}%
\special{pa 1092 2144}%
\special{pa 1040 2190}%
\special{fp}%
%
\special{pn 8}%
\special{ar 1000 2200 20 20  0.0000000 6.2831853}%
%
\special{pn 8}%
\special{pa 1430 830}%
\special{pa 1760 980}%
\special{fp}%
\special{sh 1}%
\special{pa 1760 980}%
\special{pa 1708 934}%
\special{pa 1712 958}%
\special{pa 1692 972}%
\special{pa 1760 980}%
\special{fp}%
%
\special{pn 8}%
\special{pa 1050 1030}%
\special{pa 1380 1180}%
\special{fp}%
\special{sh 1}%
\special{pa 1380 1180}%
\special{pa 1328 1134}%
\special{pa 1332 1158}%
\special{pa 1312 1172}%
\special{pa 1380 1180}%
\special{fp}%
%
\special{pn 8}%
\special{pa 1450 1230}%
\special{pa 1780 1380}%
\special{fp}%
\special{sh 1}%
\special{pa 1780 1380}%
\special{pa 1728 1334}%
\special{pa 1732 1358}%
\special{pa 1712 1372}%
\special{pa 1780 1380}%
\special{fp}%
%
\special{pn 8}%
\special{pa 1050 1430}%
\special{pa 1380 1580}%
\special{fp}%
\special{sh 1}%
\special{pa 1380 1580}%
\special{pa 1328 1534}%
\special{pa 1332 1558}%
\special{pa 1312 1572}%
\special{pa 1380 1580}%
\special{fp}%
%
\special{pn 8}%
\special{pa 1450 1630}%
\special{pa 1780 1780}%
\special{fp}%
\special{sh 1}%
\special{pa 1780 1780}%
\special{pa 1728 1734}%
\special{pa 1732 1758}%
\special{pa 1712 1772}%
\special{pa 1780 1780}%
\special{fp}%
%
\special{pn 8}%
\special{pa 1040 1830}%
\special{pa 1370 1980}%
\special{fp}%
\special{sh 1}%
\special{pa 1370 1980}%
\special{pa 1318 1934}%
\special{pa 1322 1958}%
\special{pa 1302 1972}%
\special{pa 1370 1980}%
\special{fp}%
%
\special{pn 8}%
\special{pa 1460 2030}%
\special{pa 1790 2180}%
\special{fp}%
\special{sh 1}%
\special{pa 1790 2180}%
\special{pa 1738 2134}%
\special{pa 1742 2158}%
\special{pa 1722 2172}%
\special{pa 1790 2180}%
\special{fp}%
%
\special{pn 8}%
\special{ar 1810 2200 20 20  0.0000000 6.2831853}%
%
\special{pn 8}%
\special{pa 1040 2220}%
\special{pa 1360 2380}%
\special{fp}%
\special{sh 1}%
\special{pa 1360 2380}%
\special{pa 1310 2332}%
\special{pa 1312 2356}%
\special{pa 1292 2368}%
\special{pa 1360 2380}%
\special{fp}%
%
\special{pn 8}%
\special{pa 1770 2210}%
\special{pa 1440 2380}%
\special{fp}%
\special{sh 1}%
\special{pa 1440 2380}%
\special{pa 1508 2368}%
\special{pa 1488 2356}%
\special{pa 1490 2332}%
\special{pa 1440 2380}%
\special{fp}%
%
\special{pn 8}%
\special{sh 1}%
\special{ar 1410 2400 10 10 0  6.28318530717959E+0000}%
\special{sh 1}%
\special{ar 1410 2500 10 10 0  6.28318530717959E+0000}%
\special{sh 1}%
\special{ar 1410 2610 10 10 0  6.28318530717959E+0000}%
\end{picture}%
\]

(2)\;Let $Q$ be the quiver:
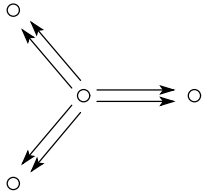
\[
\unitlength 0.1in
\begin{picture}( 10.1200,  9.7100)( 12.1000,-12.9100)
%
\special{pn 8}%
\special{ar 1610 800 32 32  0.0000000 6.2831853}%
%
\special{pn 8}%
\special{ar 2190 800 32 32  0.0000000 6.2831853}%
%
\special{pn 8}%
\special{pa 1680 770}%
\special{pa 2080 770}%
\special{fp}%
\special{sh 1}%
\special{pa 2080 770}%
\special{pa 2014 750}%
\special{pa 2028 770}%
\special{pa 2014 790}%
\special{pa 2080 770}%
\special{fp}%
%
\special{pn 8}%
\special{pa 1680 830}%
\special{pa 2080 830}%
\special{fp}%
\special{sh 1}%
\special{pa 2080 830}%
\special{pa 2014 810}%
\special{pa 2028 830}%
\special{pa 2014 850}%
\special{pa 2080 830}%
\special{fp}%
%
\special{pn 8}%
\special{ar 1242 352 32 32  0.0000000 6.2831853}%
%
\special{pn 8}%
\special{pa 1548 760}%
\special{pa 1290 456}%
\special{fp}%
\special{sh 1}%
\special{pa 1290 456}%
\special{pa 1318 520}%
\special{pa 1324 496}%
\special{pa 1348 494}%
\special{pa 1290 456}%
\special{fp}%
%
\special{pn 8}%
\special{pa 1594 720}%
\special{pa 1336 416}%
\special{fp}%
\special{sh 1}%
\special{pa 1336 416}%
\special{pa 1364 480}%
\special{pa 1370 458}%
\special{pa 1394 454}%
\special{pa 1336 416}%
\special{fp}%
%
\special{pn 8}%
\special{ar 1242 1260 32 32  0.0000000 6.2831853}%
%
\special{pn 8}%
\special{pa 1594 888}%
\special{pa 1338 1194}%
\special{fp}%
\special{sh 1}%
\special{pa 1338 1194}%
\special{pa 1396 1156}%
\special{pa 1372 1154}%
\special{pa 1366 1130}%
\special{pa 1338 1194}%
\special{fp}%
%
\special{pn 8}%
\special{pa 1548 850}%
\special{pa 1290 1156}%
\special{fp}%
\special{sh 1}%
\special{pa 1290 1156}%
\special{pa 1348 1118}%
\special{pa 1324 1114}%
\special{pa 1318 1092}%
\special{pa 1290 1156}%
\special{fp}%
\end{picture}%
\]

Then $\ptiltq(Q)$ is given by the following:
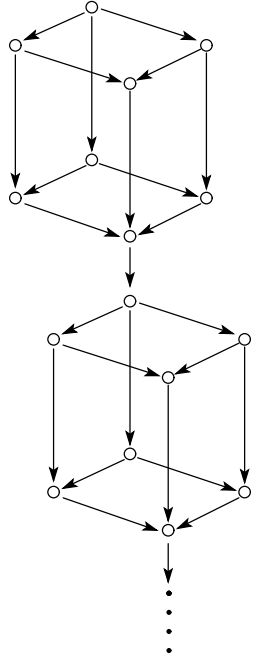
\[
\unitlength 0.1in
\begin{picture}( 12.6000, 34.0300)(  7.7000,-37.7300)
%
\special{pn 8}%
\special{ar 1200 400 30 30  0.0000000 6.2831853}%
%
\special{pn 8}%
\special{pa 1160 430}%
\special{pa 850 570}%
\special{fp}%
\special{sh 1}%
\special{pa 850 570}%
\special{pa 920 562}%
\special{pa 900 548}%
\special{pa 904 524}%
\special{pa 850 570}%
\special{fp}%
%
\special{pn 8}%
\special{ar 800 600 30 30  0.0000000 6.2831853}%
%
\special{pn 8}%
\special{ar 1800 600 30 30  0.0000000 6.2831853}%
%
\special{pn 8}%
\special{ar 1400 800 30 30  0.0000000 6.2831853}%
%
\special{pn 8}%
\special{ar 1200 1200 30 30  0.0000000 6.2831853}%
%
\special{pn 8}%
\special{ar 800 1400 30 30  0.0000000 6.2831853}%
%
\special{pn 8}%
\special{ar 1400 1600 30 30  0.0000000 6.2831853}%
%
\special{pn 8}%
\special{ar 1800 1400 30 30  0.0000000 6.2831853}%
%
\special{pn 8}%
\special{pa 850 630}%
\special{pa 1350 790}%
\special{fp}%
\special{sh 1}%
\special{pa 1350 790}%
\special{pa 1294 752}%
\special{pa 1300 774}%
\special{pa 1280 790}%
\special{pa 1350 790}%
\special{fp}%
%
\special{pn 8}%
\special{pa 1250 410}%
\special{pa 1740 570}%
\special{fp}%
\special{sh 1}%
\special{pa 1740 570}%
\special{pa 1684 530}%
\special{pa 1690 554}%
\special{pa 1670 568}%
\special{pa 1740 570}%
\special{fp}%
%
\special{pn 8}%
\special{pa 1770 620}%
\special{pa 1440 780}%
\special{fp}%
\special{sh 1}%
\special{pa 1440 780}%
\special{pa 1510 770}%
\special{pa 1488 758}%
\special{pa 1492 734}%
\special{pa 1440 780}%
\special{fp}%
%
\special{pn 8}%
\special{pa 1200 450}%
\special{pa 1200 1150}%
\special{fp}%
\special{sh 1}%
\special{pa 1200 1150}%
\special{pa 1220 1084}%
\special{pa 1200 1098}%
\special{pa 1180 1084}%
\special{pa 1200 1150}%
\special{fp}%
%
\special{pn 8}%
\special{pa 800 650}%
\special{pa 800 1340}%
\special{fp}%
\special{sh 1}%
\special{pa 800 1340}%
\special{pa 820 1274}%
\special{pa 800 1288}%
\special{pa 780 1274}%
\special{pa 800 1340}%
\special{fp}%
%
\special{pn 8}%
\special{pa 1800 640}%
\special{pa 1800 1350}%
\special{fp}%
\special{sh 1}%
\special{pa 1800 1350}%
\special{pa 1820 1284}%
\special{pa 1800 1298}%
\special{pa 1780 1284}%
\special{pa 1800 1350}%
\special{fp}%
%
\special{pn 8}%
\special{pa 1240 1220}%
\special{pa 1750 1390}%
\special{fp}%
\special{sh 1}%
\special{pa 1750 1390}%
\special{pa 1694 1350}%
\special{pa 1700 1374}%
\special{pa 1680 1388}%
\special{pa 1750 1390}%
\special{fp}%
%
\special{pn 8}%
\special{pa 1170 1230}%
\special{pa 850 1380}%
\special{fp}%
\special{sh 1}%
\special{pa 850 1380}%
\special{pa 920 1370}%
\special{pa 898 1358}%
\special{pa 902 1334}%
\special{pa 850 1380}%
\special{fp}%
%
\special{pn 8}%
\special{pa 850 1430}%
\special{pa 1340 1600}%
\special{fp}%
\special{sh 1}%
\special{pa 1340 1600}%
\special{pa 1284 1560}%
\special{pa 1290 1584}%
\special{pa 1270 1598}%
\special{pa 1340 1600}%
\special{fp}%
%
\special{pn 8}%
\special{pa 1770 1430}%
\special{pa 1450 1590}%
\special{fp}%
\special{sh 1}%
\special{pa 1450 1590}%
\special{pa 1520 1578}%
\special{pa 1498 1566}%
\special{pa 1502 1542}%
\special{pa 1450 1590}%
\special{fp}%
%
\special{pn 8}%
\special{pa 1400 840}%
\special{pa 1400 1550}%
\special{fp}%
\special{sh 1}%
\special{pa 1400 1550}%
\special{pa 1420 1484}%
\special{pa 1400 1498}%
\special{pa 1380 1484}%
\special{pa 1400 1550}%
\special{fp}%
%
\special{pn 8}%
\special{ar 1400 1940 30 30  0.0000000 6.2831853}%
%
\special{pn 8}%
\special{pa 1360 1970}%
\special{pa 1050 2110}%
\special{fp}%
\special{sh 1}%
\special{pa 1050 2110}%
\special{pa 1120 2102}%
\special{pa 1100 2088}%
\special{pa 1104 2064}%
\special{pa 1050 2110}%
\special{fp}%
%
\special{pn 8}%
\special{ar 1000 2140 30 30  0.0000000 6.2831853}%
%
\special{pn 8}%
\special{ar 2000 2140 30 30  0.0000000 6.2831853}%
%
\special{pn 8}%
\special{ar 1600 2340 30 30  0.0000000 6.2831853}%
%
\special{pn 8}%
\special{ar 1400 2740 30 30  0.0000000 6.2831853}%
%
\special{pn 8}%
\special{ar 1000 2940 30 30  0.0000000 6.2831853}%
%
\special{pn 8}%
\special{ar 1600 3140 30 30  0.0000000 6.2831853}%
%
\special{pn 8}%
\special{ar 2000 2940 30 30  0.0000000 6.2831853}%
%
\special{pn 8}%
\special{pa 1050 2170}%
\special{pa 1550 2330}%
\special{fp}%
\special{sh 1}%
\special{pa 1550 2330}%
\special{pa 1494 2292}%
\special{pa 1500 2314}%
\special{pa 1480 2330}%
\special{pa 1550 2330}%
\special{fp}%
%
\special{pn 8}%
\special{pa 1450 1950}%
\special{pa 1940 2110}%
\special{fp}%
\special{sh 1}%
\special{pa 1940 2110}%
\special{pa 1884 2070}%
\special{pa 1890 2094}%
\special{pa 1870 2108}%
\special{pa 1940 2110}%
\special{fp}%
%
\special{pn 8}%
\special{pa 1970 2160}%
\special{pa 1640 2320}%
\special{fp}%
\special{sh 1}%
\special{pa 1640 2320}%
\special{pa 1710 2310}%
\special{pa 1688 2298}%
\special{pa 1692 2274}%
\special{pa 1640 2320}%
\special{fp}%
%
\special{pn 8}%
\special{pa 1400 1990}%
\special{pa 1400 2690}%
\special{fp}%
\special{sh 1}%
\special{pa 1400 2690}%
\special{pa 1420 2624}%
\special{pa 1400 2638}%
\special{pa 1380 2624}%
\special{pa 1400 2690}%
\special{fp}%
%
\special{pn 8}%
\special{pa 1000 2190}%
\special{pa 1000 2880}%
\special{fp}%
\special{sh 1}%
\special{pa 1000 2880}%
\special{pa 1020 2814}%
\special{pa 1000 2828}%
\special{pa 980 2814}%
\special{pa 1000 2880}%
\special{fp}%
%
\special{pn 8}%
\special{pa 2000 2180}%
\special{pa 2000 2890}%
\special{fp}%
\special{sh 1}%
\special{pa 2000 2890}%
\special{pa 2020 2824}%
\special{pa 2000 2838}%
\special{pa 1980 2824}%
\special{pa 2000 2890}%
\special{fp}%
%
\special{pn 8}%
\special{pa 1440 2760}%
\special{pa 1950 2930}%
\special{fp}%
\special{sh 1}%
\special{pa 1950 2930}%
\special{pa 1894 2890}%
\special{pa 1900 2914}%
\special{pa 1880 2928}%
\special{pa 1950 2930}%
\special{fp}%
%
\special{pn 8}%
\special{pa 1370 2770}%
\special{pa 1050 2920}%
\special{fp}%
\special{sh 1}%
\special{pa 1050 2920}%
\special{pa 1120 2910}%
\special{pa 1098 2898}%
\special{pa 1102 2874}%
\special{pa 1050 2920}%
\special{fp}%
%
\special{pn 8}%
\special{pa 1050 2970}%
\special{pa 1540 3140}%
\special{fp}%
\special{sh 1}%
\special{pa 1540 3140}%
\special{pa 1484 3100}%
\special{pa 1490 3124}%
\special{pa 1470 3138}%
\special{pa 1540 3140}%
\special{fp}%
%
\special{pn 8}%
\special{pa 1970 2970}%
\special{pa 1650 3130}%
\special{fp}%
\special{sh 1}%
\special{pa 1650 3130}%
\special{pa 1720 3118}%
\special{pa 1698 3106}%
\special{pa 1702 3082}%
\special{pa 1650 3130}%
\special{fp}%
%
\special{pn 8}%
\special{pa 1600 2380}%
\special{pa 1600 3090}%
\special{fp}%
\special{sh 1}%
\special{pa 1600 3090}%
\special{pa 1620 3024}%
\special{pa 1600 3038}%
\special{pa 1580 3024}%
\special{pa 1600 3090}%
\special{fp}%
%
\special{pn 8}%
\special{sh 1}%
\special{ar 1604 3470 10 10 0  6.28318530717959E+0000}%
\special{sh 1}%
\special{ar 1604 3470 10 10 0  6.28318530717959E+0000}%
\special{sh 1}%
\special{ar 1604 3470 10 10 0  6.28318530717959E+0000}%
\special{sh 1}%
\special{ar 1604 3470 10 10 0  6.28318530717959E+0000}%
%
\special{pn 8}%
\special{sh 1}%
\special{ar 1604 3570 10 10 0  6.28318530717959E+0000}%
%
\special{pn 8}%
\special{sh 1}%
\special{ar 1604 3670 10 10 0  6.28318530717959E+0000}%
%
\special{pn 8}%
\special{sh 1}%
\special{ar 1604 3770 10 10 0  6.28318530717959E+0000}%
%
\special{pn 8}%
\special{pa 1400 1660}%
\special{pa 1400 1860}%
\special{fp}%
\special{sh 1}%
\special{pa 1400 1860}%
\special{pa 1420 1794}%
\special{pa 1400 1808}%
\special{pa 1380 1794}%
\special{pa 1400 1860}%
\special{fp}%
%
\special{pn 8}%
\special{pa 1600 3210}%
\special{pa 1600 3410}%
\special{fp}%
\special{sh 1}%
\special{pa 1600 3410}%
\special{pa 1620 3344}%
\special{pa 1600 3358}%
\special{pa 1580 3344}%
\special{pa 1600 3410}%
\special{fp}%
\end{picture}%
\]

(3)\;Let $Q$ be the quiver:
\[
\unitlength 0.1in
\begin{picture}( 11.3000,  0.6500)( 11.7000,-10.3000)
%
\special{pn 8}%
\special{ar 1200 1000 30 30  0.0000000 6.2831853}%
%
\special{pn 8}%
\special{pa 1260 990}%
\special{pa 1660 990}%
\special{fp}%
\special{sh 1}%
\special{pa 1660 990}%
\special{pa 1594 970}%
\special{pa 1608 990}%
\special{pa 1594 1010}%
\special{pa 1660 990}%
\special{fp}%
%
\special{pn 8}%
\special{ar 1720 1000 30 30  0.0000000 6.2831853}%
%
\special{pn 8}%
\special{pa 1780 970}%
\special{pa 2180 970}%
\special{fp}%
\special{sh 1}%
\special{pa 2180 970}%
\special{pa 2114 950}%
\special{pa 2128 970}%
\special{pa 2114 990}%
\special{pa 2180 970}%
\special{fp}%
%
\special{pn 8}%
\special{pa 1780 1020}%
\special{pa 2180 1020}%
\special{fp}%
\special{sh 1}%
\special{pa 2180 1020}%
\special{pa 2114 1000}%
\special{pa 2128 1020}%
\special{pa 2114 1040}%
\special{pa 2180 1020}%
\special{fp}%
%
\special{pn 8}%
\special{ar 2270 1000 30 30  0.0000000 6.2831853}%
\end{picture}%
\]

Then $\ptiltq(Q)$ is given by the following:
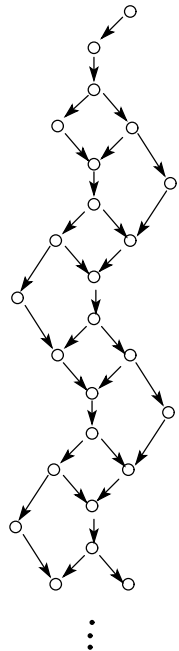
\[\input{lat-example-3}\]

\end{exmp}

\begin{lem}
\label{join-irre}
Assume that $Q$ satisfies the condition $\mathrm{(C)}$. Then the set of join-irreducible elements of $\mathcal{T}_{\mathrm{p}}(Q)$ is $\{\tau^{-r}T(a)\mid a\in Q_{0},\ r\in \mathbb{Z}_{\geq 0}^{Q_{0}} \}$.

\end{lem}

\begin{proof}
Theorem\;\ref{t0} implies $T\in \mathcal{T}_{\mathrm{p}}(Q)$ is join-irreducible if and only if there is the unique direct successor of $T$ in $\ptiltq(Q)$. Let $T$ be a join irreducible element and $T^{'}$ be the its direct successor. Let $a\in Q_{0}$ such that $T^{'}_{a}=T_{a}+1$. Without loss of generality, we can assume that $T_{x}=0$ for some $x\in Q_{0}$. Then it is suffice to show that $T=T(a)$. 

We now define a partial order $\leq_{Q}$ on $Q_{0}$ as follows:
\[x\leq_{Q} y\stackrel{\mathrm{def}}{\Leftrightarrow} \mathrm{there\ exists\ a\ path\ from\ }x\ \mathrm{to\ }y.\]

Let $b\in Q_{0}$ be a minimal element of $\{y\in Q_{0}\mid T_{y}=0\}$. Then $T^{''}:=\bigoplus_{x\neq b}\tau^{-T_{x}}P(x)\oplus \tau^{-1}P(b)\in \mathcal{T}_{\mathrm{p}}(Q)$. Therefore we obtain $T_{a}=0$. In particular $T\geq T(a)$. Suppose that $T>T(a)$. Then there is a path 
\[T\rightarrow T^{1}\rightarrow T^{2}\rightarrow\cdots \rightarrow T^{r}=T(a).\]
Since $T_{a}=0=T(a)_{a}$, we have $T^{1}\neq T^{'}$. We now get a contradiction.      

\end{proof}
\begin{defn}
We define a poset $J=J(Q)$ as follows:\\
$\bullet\;J=\mathbb{Z}_{\geq 0}\times Q_{0}$ as a set.\\    $\bullet\;(r,a)\leq (s,b)\stackrel{\mathrm{def}}{\Leftrightarrow} l_{Q}(a,x)+r\geq l_{Q}(b,x)+s$ for any $x\in Q_{0}$.

\end{defn}
We set $T(\mathbf{j}):=\tau^{-r}T(x)$ for any $\mathbf{j}=(r,x)\in J$. Note that \[\mathbf{j}_{1}\leq \mathbf{j}_{2}\Leftrightarrow T(\mathbf{j}_{1})\leq T(\mathbf{j}_{2}).\] 

\begin{cor}
Assume that $Q$ satisfies the condition $(\mathrm{C})$. Then 
a map $\rho:\mathcal{I}(Q)\setminus\{\emptyset\}\ni I \mapsto \bigvee_{\mathbf{i}\in I}T(\mathbf{i})\in \mathcal{T}_{\mathrm{p}}(Q)$ induces a poset isomorphism
\[\mathcal{I}(Q)\setminus\{\emptyset\}\simeq \mathcal{T}_{\mathrm{p}}(Q),\]
where $\mathcal{I}(Q)$ be a ideal-poset of $J(Q)$.
\end{cor}

\begin{proof}
Let $I\in\mathcal{I}(Q)\setminus\{\emptyset\}$. Then it is easy to check that there is a finite subset $\{\mathbf{i}_{1},\cdots \mathbf{i}_{m}\}$ of $I$ such that $I=\{\mathbf{j}\in J\mid \mathbf{j}\leq \mathbf{i}_{t}\ \mathrm{for\ some\ }t \}$. Then $\bigvee_{\mathbf{i}\in I} T(\mathbf{i})=\bigvee_{t=1}^{m}T(\mathbf{i}_{t})$. In particular 
a map $\rho:\mathcal{I}(Q)\setminus\{\emptyset\}\ni I\mapsto \bigvee_{\mathbf{i}\in I} T(\mathbf{i})\in \mathcal{T}_{\mathrm{p}}(Q)$ well-defined. It is obvious that $\rho$ is an order-preserving
map. Let $I,I^{'}\in \mathcal{I}(Q)\setminus\{\emptyset\}$ with $\rho(I)\leq \rho(I^{'})$ and $(r,x)\in I$. Then $T((r,x))\leq \bigvee_{\mathbf{i}\in I} T(\mathbf{i})\leq \bigvee_{\mathbf{i}\in I^{'}} T(\mathbf{i})$ implies $r^{'}+l_{Q}(x^{'},x)\leq r$ for some $\mathbf{i}^{'}:=(r^{'},x^{'})\in I^{'}$.
Since $T((r,x))$ is the minimum element of $\{T\in \mathcal{T}_{\mathrm{p}}(Q)\mid T_{x}\leq r\}$, we obtain $T((r,x))\leq T(\mathbf{i}^{'})$. Therefore we have $(r,x)\leq \mathbf{i}^{'}$. In particular we obtain $(r,x)\in I^{'}$.

We  show that $\rho$ is bijection. If $\rho(I)=\rho(I^{'})$, then $I\subset I^{'}$ and $I^{'}\subset I$. Therefore $\rho$ is injection. Let $T\in \mathcal{T}_{\mathrm{p}}(Q)$. Then it is easy to check that $T=\bigvee_{x\in Q_{0}}\tau^{-T_{x}}T(x)$. Indeed $\tau^{-T_{x}}T(x)$ is the minimum element of $\{T^{'}\mid T^{'}_{x}\leq T_{x}\}$. Therefore we obtain $T\geq \tau^{-T_{x}}T(x)$ for any $x\in Q_{0}$. In particular we have $T\geq \bigvee_{x\in Q_{0}}\tau^{-T_{x}}T(x)$. Since $(\bigvee_{x\in Q_{0}}\tau^{-T_{x}}T(x))_{a}\leq T_{a}$ for any $a\in Q_{0}$, we obtain $T\leq \bigvee_{x\in Q_{0}}\tau^{-T_{x}}T(x)$. Therefore we obtain $\rho(I)=T$ for $I:=\{\mathbf{j}\in J\mid \mathbf{j}\leq (T_{x},x)\ \mathrm{for\ some\ }x\in Q_{0}\}$. In particular $\rho$ is bijection.

\end{proof}

\begin{lem}
For $\mathbf{j}=(r,a)\in \mathbb{Z}_{\geq 0}\times Q_{0}$,  set $P(\mathbf{j}):=\tau^{-r}P(a)$.
Then $\mathbf{j}_{1}\leq \mathbf{j}_{2}$ if and only if there is a path from $P(\mathbf{j}_{2})$ to $P(\mathbf{j}_{1})$ 
in $\mathrm{Auslander}$-$\mathrm{Reiten}$ quiver of $kQ$. 

\end{lem}
\begin{proof}
Let $\mathbf{j}_{1}=(r,a)$ and $\mathbf{j}_{2}=(s,b)$. First we assume that there exists an arrow $P(\mathbf{j}_{2})\rightarrow P(\mathbf{j}_{1})$ in $\Gamma(kQ)$. Then we have $(1)$\;$a\rightarrow b$ in $Q$ and $r=s$ or $(2)$\;$b\rightarrow a$ and $r=s+1$.  In both of two cases, we have $l_{Q}(a,x)+r\geq l_{Q}(b,x)+s$ for any $x\in Q_{0}$.

Next we assume that $\mathbf{j}_{1}\leq \mathbf{j}_{2}$ and let $t:=l_{Q}(b,a)$. Then we have  $r\geq s+t$. By definition of $l_{Q}$, we can take a sub-quiver
\[b\leftarrow \cdots \leftarrow b_{1}\rightarrow a_{1}\leftarrow\cdots \leftarrow b_{2}\rightarrow a_{2}\cdots b_{t}\rightarrow a_{t}\leftarrow \cdots \leftarrow a \] of $Q$.
In particular we obtain a path
\[P(\mathbf{j}_{2})\rightarrow \cdots \rightarrow \tau^{-s}P(b_{1})\rightarrow \tau^{-s-1}P(a_{1})\rightarrow\cdots \rightarrow \tau^{-s-t}P(a_{t})\rightarrow \cdots \rightarrow \tau^{-s-t}P(a)\]
in $\Gamma(kQ)$.
Now $r\geq s+t$ implies that there is a path from $\tau^{-s-t}P(a)$ to $\tau^{-r}P(a)=P(\mathbf{j}_{1})$ in $\Gamma(kQ)$.
\end{proof}

\begin{defn}
For any acyclic quiver $\Gamma$, we define a poset $\mathcal{P}(\Gamma)$ as follows:\\
$\bullet$\;$\mathcal{P}(\Gamma)=\Gamma_{0}$ as a set.\\
$\bullet$\;$x\leq y$ if there is a path from $y$ to $x$ in $\Gamma$.

\end{defn}

\begin{cor}
Let $\Gamma_{\mathrm{p}}(Q)$ be the pre-projective component of $\mathrm{Auslander}$-$\mathrm{Reiten}$ quiver of $kQ$. 
Then  the poset $\mathcal{T}_{\mathrm{p}}(Q)$ is isomorphic to $\mathcal{I}(\mathcal{P}(\Gamma_{\mathrm{p}}(Q)))\setminus\{\emptyset\}$.
\end{cor}

\begin{exmp}
Let $Q$ be a quiver:
\[
\unitlength 0.1in
\begin{picture}( 17.5000,  1.5000)( 10.0000, -7.8000)
\put(10.0000,-8.1000){\makebox(0,0)[lb]{$Q:$}}%
\put(13.5000,-8.0000){\makebox(0,0)[lb]{$0$}}%
%
\special{pn 8}%
\special{pa 1560 776}%
\special{pa 1960 776}%
\special{fp}%
\special{sh 1}%
\special{pa 1960 776}%
\special{pa 1894 756}%
\special{pa 1908 776}%
\special{pa 1894 796}%
\special{pa 1960 776}%
\special{fp}%
%
\special{pn 8}%
\special{pa 1560 726}%
\special{pa 1960 726}%
\special{fp}%
\special{sh 1}%
\special{pa 1960 726}%
\special{pa 1894 706}%
\special{pa 1908 726}%
\special{pa 1894 746}%
\special{pa 1960 726}%
\special{fp}%
\put(20.5000,-8.1000){\makebox(0,0)[lb]{$1$}}%
%
\special{pn 8}%
\special{pa 2260 776}%
\special{pa 2660 776}%
\special{fp}%
\special{sh 1}%
\special{pa 2660 776}%
\special{pa 2594 756}%
\special{pa 2608 776}%
\special{pa 2594 796}%
\special{pa 2660 776}%
\special{fp}%
%
\special{pn 8}%
\special{pa 2260 726}%
\special{pa 2660 726}%
\special{fp}%
\special{sh 1}%
\special{pa 2660 726}%
\special{pa 2594 706}%
\special{pa 2608 726}%
\special{pa 2594 746}%
\special{pa 2660 726}%
\special{fp}%
\put(27.5000,-8.1000){\makebox(0,0)[lb]{$2$}}%
\end{picture}%
\] 
Then $\mathcal{I}(\mathcal{P}(\Gamma_{\mathrm{p}}(kQ)))\setminus\{\emptyset\}$ is given by the following:
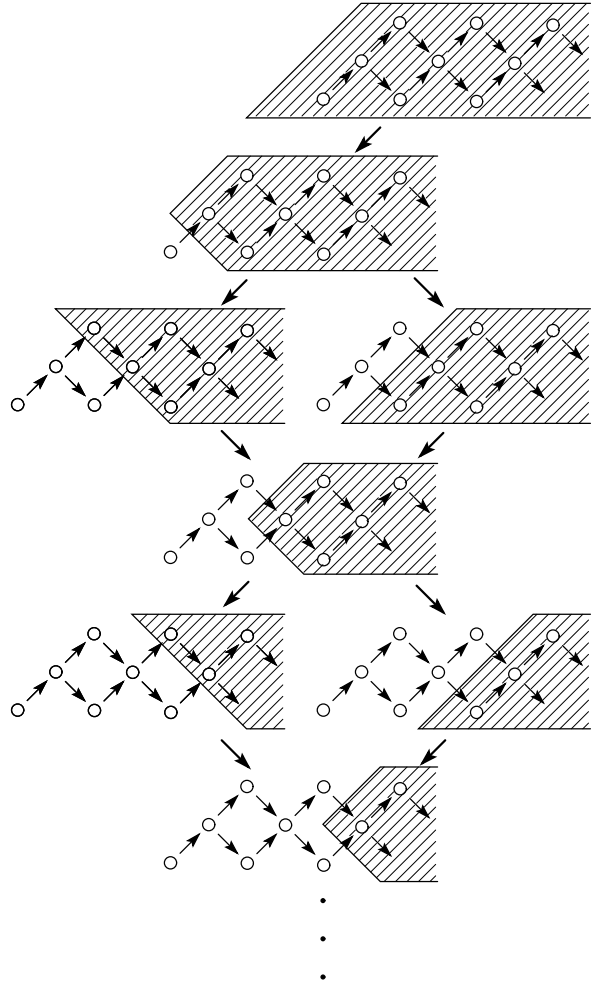
\[\input{dr-example-4}\]
\end{exmp}

Let $L$ be an infinite distributive lattice with the maximum element $o$ and $\tau^{-1}$ be a lattice inclusion $L\rightarrow L$ which induces a quiver inclusion
 $\overrightarrow{L}\rightarrow \overrightarrow{L}$. 
\begin{defn}
Let $\sim$ be an equivalence relation on $\overrightarrow{L}_{1}$ generated by the following:\\
(a)$\alpha\sim \tau^{-1}\alpha$.\\
(b)$\alpha\sim \beta$ if there is a full sub-quiver $S(\alpha,\beta)$ of $\overrightarrow{L}$. 
\[
\unitlength 0.1in
\begin{picture}( 13.8500,  9.2000)( 15.1000, -9.9000)
\put(21.7000,-3.6000){\makebox(0,0)[lb]{$s(\alpha)$}}%
%
\special{pn 8}%
\special{pa 2480 300}%
\special{pa 2750 300}%
\special{fp}%
\special{sh 1}%
\special{pa 2750 300}%
\special{pa 2684 280}%
\special{pa 2698 300}%
\special{pa 2684 320}%
\special{pa 2750 300}%
\special{fp}%
\put(28.3000,-3.6000){\makebox(0,0)[lb]{$t(\alpha)$}}%
\put(25.9000,-2.4000){\makebox(0,0)[lb]{$\alpha$}}%
%
\special{pn 8}%
\special{pa 2320 420}%
\special{pa 2320 790}%
\special{fp}%
\special{sh 1}%
\special{pa 2320 790}%
\special{pa 2340 724}%
\special{pa 2320 738}%
\special{pa 2300 724}%
\special{pa 2320 790}%
\special{fp}%
%
\special{pn 8}%
\special{pa 2490 920}%
\special{pa 2760 920}%
\special{fp}%
\special{sh 1}%
\special{pa 2760 920}%
\special{pa 2694 900}%
\special{pa 2708 920}%
\special{pa 2694 940}%
\special{pa 2760 920}%
\special{fp}%
\put(21.8000,-9.6000){\makebox(0,0)[lb]{$s(\beta)$}}%
\put(28.3000,-9.8000){\makebox(0,0)[lb]{$t(\beta)$}}%
\put(25.9000,-11.6000){\makebox(0,0)[lb]{$\beta$}}%
%
\special{pn 8}%
\special{pa 2890 420}%
\special{pa 2890 790}%
\special{fp}%
\special{sh 1}%
\special{pa 2890 790}%
\special{pa 2910 724}%
\special{pa 2890 738}%
\special{pa 2870 724}%
\special{pa 2890 790}%
\special{fp}%
\put(15.1000,-6.7000){\makebox(0,0)[lb]{$S(\alpha,\beta)$}}%
\end{picture}%
\] 
 
Then we put $\Lambda=\Lambda(L,\tau^{-1}):=\overrightarrow{L}_{1}/\!\raisebox{-2pt}{$\sim$}$.
\end{defn}

Let $P:=L\setminus \tau^{-1}L$.
\begin{prop}
\label{p1}
Assume that $(L,\tau^{-1})$ satisfies the following conditions,\\
$(\mathrm{c}_{0})$\;$P\neq \emptyset$ is finite.\\
$(\mathrm{c}_{1})$\;$L=\coprod_{r\geq 0}\tau^{-r}P$.\\
$(\mathrm{c}_{2})$\;$x\not< \tau^{-r}y$ for any $x,y\in P$ and $r>0$.        \\
$(\mathrm{c}_{3})$\;$x\rightarrow y$ implies $y\geq \tau^{-1}x$.\\
$(\mathrm{c}_{4})$\;For any $x\in P_{0}$ there exists a path $w:x\stackrel{\alpha_{1}}{\rightarrow}x_{1}\stackrel{\alpha_{2}}{\rightarrow}\cdots \stackrel{\alpha_{r}}{\rightarrow}\tau^{-1}x$
in $\overrightarrow{L}$ such that $\{\alpha_{i}\}_{1\leq i\leq r}$ is a minimal representable of $\Lambda$.

Then the following assertions hold.\\
$(1)$\; If $x>y$ in $L$, then there is a path from $x$ to $y$ in $\overrightarrow{L}$. In particular $\overrightarrow{L}$ is a connected quiver. \\
$(2)$\; For any path $w:x_{0}\stackrel{\alpha_{1}}{\rightarrow}x_{1}\stackrel{\alpha_{2}}{\rightarrow}\cdots \stackrel{\alpha_{r}}{\rightarrow}x_{r}$ in $\overrightarrow{L}$ and $\lambda\in \Lambda$ we put $\phi(w,\lambda):=\#\{i\mid \alpha_{i}/\!\raisebox{-2pt}{$\sim$}=\lambda\}$. If $s(w)=s(w^{'})$ and $t(w)=t(w^{'})$, then 
$\phi(w,\lambda)=\phi(w^{'},\lambda)$.\\
$(3)$\;For any $x\in L$ we put $\phi(x):=(\phi(w,\lambda))_{\lambda\in \Lambda}$, where $w$ is a path from $o$ to $x$.
Then $\phi$ induces a quiver inclusion from  $\overrightarrow{L}$ to the Hasse-quiver of the poset $(\mathbb{Z}_{\geq 0}^{\Lambda},\leq^{\mathrm{op}})$.\\
$(4)$\;Let $x,y\in L$. Then $x<y$ if and only if $\phi(x)\leq^{\mathrm{op}}\phi(y)$   \\
$(5)$\;Let $x,y\in L$ and $\lambda\in \Lambda$. Then $\phi(x\vee y)_{\lambda}=\mathrm{min}\{\phi(x)_{\lambda},\phi(y)_{\lambda}\}$ and $\phi(x\wedge y)_{\lambda}=\mathrm{max}\{\phi(x)_{\lambda},\phi(y)_{\lambda}\}$.
\end{prop}
\begin{proof}
Let  $L(x):=\{y\in L_{0}\mid y\geq x\}$. We note that $L(x)$ is a distributive lattice and its Hasse-quiver $\overrightarrow{L}(x)$ is a full sub-quiver of $\overrightarrow{L}$. We claim that $L(x)$ is finite. Indeed the condition\;$(\mathrm{c}_{1})$ implies there exists $r\geq 0$ such that $x\in \tau^{-r}P$ and then the condition $(c_{2})$ implies $L(x)\subset \coprod_{i=0}^{r}\tau^{-i}P$. Therefore the condition $(c_{0})$ implies $L(x)$ is finite.  In particular $L(x)$ is a finite distributive lattice for any $x\in L$.
 \\
(1) Let $x,y\in L$ with $x>y$. Since $\overrightarrow{L}(y)$ is a finite full sub-quiver of $\overrightarrow{L}$ and $x\in L(y)$, there is a path from $x$ to $y$ in $\overrightarrow{L}$.\\     
(2) Let $w:x\stackrel{\alpha_{1}}{\rightarrow}x_{1}\stackrel{\alpha_{2}}{\rightarrow}\cdots \stackrel{\alpha_{r}}{\rightarrow}x_{r}=y$ and $w^{'}:x\stackrel{\beta_{1}}{\rightarrow}x^{'}_{1}\stackrel{\beta_{2}}{\rightarrow}\cdots \stackrel{\beta_{r}}{\rightarrow}x^{'}_{r}=y$.
We prove with the using of an induction on $l(w)=l(w^{'})=r$.

$(r=1)$ In this case the assertion is obvious.

$(r>1)$ Without loss of generality, we can assume  $x_{1}\neq x^{'}_{1}$. Put $s=\mathrm{min}\{i\mid x_{i+1}\leq x^{'}_{1}\}$ then $x_{s}\wedge x^{'}_{1}=x_{s+1}$. We put \[x^{''}_{i}:=\left\{\begin{array}{ll}
x_{i-1}\wedge x^{'}_{1} & \mathrm{if}\ i\leq s  \\ 
x_{i} & \mathrm{if}\ i\geq s+1
\end{array}\right.\]
Then we get the following diagram,
\[
\unitlength 0.1in
\begin{picture}( 52.0000,  8.7000)(  3.7000, -9.8000)
\put(7.4000,-3.9000){\makebox(0,0)[lb]{$x$}}%
\put(12.8000,-4.0000){\makebox(0,0)[lb]{$x_{1}$}}%
\put(18.6000,-4.0000){\makebox(0,0)[lb]{$x_{2}$}}%
%
\special{pn 8}%
\special{pa 900 340}%
\special{pa 1200 340}%
\special{fp}%
\special{sh 1}%
\special{pa 1200 340}%
\special{pa 1134 320}%
\special{pa 1148 340}%
\special{pa 1134 360}%
\special{pa 1200 340}%
\special{fp}%
%
\special{pn 8}%
\special{pa 1520 340}%
\special{pa 1820 340}%
\special{fp}%
\special{sh 1}%
\special{pa 1820 340}%
\special{pa 1754 320}%
\special{pa 1768 340}%
\special{pa 1754 360}%
\special{pa 1820 340}%
\special{fp}%
%
\special{pn 8}%
\special{pa 2040 340}%
\special{pa 2340 340}%
\special{fp}%
\special{sh 1}%
\special{pa 2340 340}%
\special{pa 2274 320}%
\special{pa 2288 340}%
\special{pa 2274 360}%
\special{pa 2340 340}%
\special{fp}%
%
\special{pn 8}%
\special{sh 1}%
\special{ar 2470 340 10 10 0  6.28318530717959E+0000}%
\special{sh 1}%
\special{ar 2540 340 10 10 0  6.28318530717959E+0000}%
\special{sh 1}%
\special{ar 2610 340 10 10 0  6.28318530717959E+0000}%
\put(29.9000,-3.9000){\makebox(0,0)[lb]{$x_{s-1}$}}%
\put(37.4000,-4.0000){\makebox(0,0)[lb]{$x_{s}$}}%
%
\special{pn 8}%
\special{pa 3350 340}%
\special{pa 3650 340}%
\special{fp}%
\special{sh 1}%
\special{pa 3650 340}%
\special{pa 3584 320}%
\special{pa 3598 340}%
\special{pa 3584 360}%
\special{pa 3650 340}%
\special{fp}%
%
\special{pn 8}%
\special{pa 2690 340}%
\special{pa 2690 340}%
\special{fp}%
%
\special{pn 8}%
\special{pa 2670 340}%
\special{pa 2970 340}%
\special{fp}%
\special{sh 1}%
\special{pa 2970 340}%
\special{pa 2904 320}%
\special{pa 2918 340}%
\special{pa 2904 360}%
\special{pa 2970 340}%
\special{fp}%
\put(7.4000,-9.9000){\makebox(0,0)[lb]{$x^{''}_{1}$}}%
\put(12.8000,-10.0000){\makebox(0,0)[lb]{$x^{''}_{2}$}}%
\put(18.6000,-10.0000){\makebox(0,0)[lb]{$x^{''}_{3}$}}%
%
\special{pn 8}%
\special{pa 900 940}%
\special{pa 1200 940}%
\special{fp}%
\special{sh 1}%
\special{pa 1200 940}%
\special{pa 1134 920}%
\special{pa 1148 940}%
\special{pa 1134 960}%
\special{pa 1200 940}%
\special{fp}%
%
\special{pn 8}%
\special{pa 1520 940}%
\special{pa 1820 940}%
\special{fp}%
\special{sh 1}%
\special{pa 1820 940}%
\special{pa 1754 920}%
\special{pa 1768 940}%
\special{pa 1754 960}%
\special{pa 1820 940}%
\special{fp}%
%
\special{pn 8}%
\special{pa 2040 940}%
\special{pa 2340 940}%
\special{fp}%
\special{sh 1}%
\special{pa 2340 940}%
\special{pa 2274 920}%
\special{pa 2288 940}%
\special{pa 2274 960}%
\special{pa 2340 940}%
\special{fp}%
%
\special{pn 8}%
\special{sh 1}%
\special{ar 2470 940 10 10 0  6.28318530717959E+0000}%
\special{sh 1}%
\special{ar 2540 940 10 10 0  6.28318530717959E+0000}%
\special{sh 1}%
\special{ar 2610 940 10 10 0  6.28318530717959E+0000}%
\put(30.7000,-10.0000){\makebox(0,0)[lb]{$x^{''}_{s}$}}%
\put(37.6000,-10.1000){\makebox(0,0)[lb]{$x^{''}_{s+1}$}}%
\put(44.3000,-9.9000){\makebox(0,0)[lb]{$x^{''}_{s+2}$}}%
%
\special{pn 8}%
\special{pa 3380 930}%
\special{pa 3680 930}%
\special{fp}%
\special{sh 1}%
\special{pa 3680 930}%
\special{pa 3614 910}%
\special{pa 3628 930}%
\special{pa 3614 950}%
\special{pa 3680 930}%
\special{fp}%
%
\special{pn 8}%
\special{pa 4030 930}%
\special{pa 4330 930}%
\special{fp}%
\special{sh 1}%
\special{pa 4330 930}%
\special{pa 4264 910}%
\special{pa 4278 930}%
\special{pa 4264 950}%
\special{pa 4330 930}%
\special{fp}%
%
\special{pn 8}%
\special{pa 4750 940}%
\special{pa 5050 940}%
\special{fp}%
\special{sh 1}%
\special{pa 5050 940}%
\special{pa 4984 920}%
\special{pa 4998 940}%
\special{pa 4984 960}%
\special{pa 5050 940}%
\special{fp}%
%
\special{pn 8}%
\special{sh 1}%
\special{ar 5130 930 10 10 0  6.28318530717959E+0000}%
\special{sh 1}%
\special{ar 5200 930 10 10 0  6.28318530717959E+0000}%
\special{sh 1}%
\special{ar 5270 930 10 10 0  6.28318530717959E+0000}%
%
\special{pn 8}%
\special{pa 2690 940}%
\special{pa 2690 940}%
\special{fp}%
%
\special{pn 8}%
\special{pa 2670 940}%
\special{pa 2970 940}%
\special{fp}%
\special{sh 1}%
\special{pa 2970 940}%
\special{pa 2904 920}%
\special{pa 2918 940}%
\special{pa 2904 960}%
\special{pa 2970 940}%
\special{fp}%
%
\special{pn 8}%
\special{pa 5340 930}%
\special{pa 5540 930}%
\special{fp}%
\special{sh 1}%
\special{pa 5540 930}%
\special{pa 5474 910}%
\special{pa 5488 930}%
\special{pa 5474 950}%
\special{pa 5540 930}%
\special{fp}%
\put(55.7000,-9.8000){\makebox(0,0)[lb]{$x^{''}_{r}=y$}}%
%
\special{pn 8}%
\special{pa 800 420}%
\special{pa 800 830}%
\special{fp}%
\special{sh 1}%
\special{pa 800 830}%
\special{pa 820 764}%
\special{pa 800 778}%
\special{pa 780 764}%
\special{pa 800 830}%
\special{fp}%
%
\special{pn 8}%
\special{pa 1310 420}%
\special{pa 1310 830}%
\special{fp}%
\special{sh 1}%
\special{pa 1310 830}%
\special{pa 1330 764}%
\special{pa 1310 778}%
\special{pa 1290 764}%
\special{pa 1310 830}%
\special{fp}%
%
\special{pn 8}%
\special{pa 1900 420}%
\special{pa 1900 830}%
\special{fp}%
\special{sh 1}%
\special{pa 1900 830}%
\special{pa 1920 764}%
\special{pa 1900 778}%
\special{pa 1880 764}%
\special{pa 1900 830}%
\special{fp}%
%
\special{pn 8}%
\special{pa 3100 430}%
\special{pa 3100 840}%
\special{fp}%
\special{sh 1}%
\special{pa 3100 840}%
\special{pa 3120 774}%
\special{pa 3100 788}%
\special{pa 3080 774}%
\special{pa 3100 840}%
\special{fp}%
%
\special{pn 8}%
\special{pa 3790 420}%
\special{pa 3790 830}%
\special{fp}%
\special{sh 1}%
\special{pa 3790 830}%
\special{pa 3810 764}%
\special{pa 3790 778}%
\special{pa 3770 764}%
\special{pa 3790 830}%
\special{fp}%
\put(3.7000,-9.8000){\makebox(0,0)[lb]{$x^{'}_{1}=$}}%
\put(9.7000,-3.0000){\makebox(0,0)[lb]{$\alpha_{1}$}}%
\put(16.0000,-2.9000){\makebox(0,0)[lb]{$\alpha_{2}$}}%
\put(34.4000,-2.8000){\makebox(0,0)[lb]{$\alpha_{s}$}}%
\put(5.6000,-6.4000){\makebox(0,0)[lb]{$\gamma_{1}$}}%
\put(9.7000,-11.3000){\makebox(0,0)[lb]{$\gamma_{2}$}}%
\put(15.9000,-11.5000){\makebox(0,0)[lb]{$\gamma_{3}$}}%
\put(33.7000,-11.3000){\makebox(0,0)[lb]{$\gamma_{s+1}$}}%
\put(38.8000,-6.8000){\makebox(0,0)[lb]{$\alpha_{s+1}$}}%
\end{picture}%
\]\\\\
in $\overrightarrow{L}$.
We consider a path $w^{''}:x\stackrel{\gamma_{1}}{\rightarrow}x^{'}_{1}= x^{''}_{1}\stackrel{\gamma_{2}}{\rightarrow} \cdots \stackrel{\gamma_{r}}{\rightarrow} x^{''}_{r}=y$. By hypothesis of induction we get $\phi(w^{'},\lambda)=\phi(w^{''},\lambda)$ for any $\lambda\in \Lambda$.
Therefore it is sufficient to show $\phi(w,\lambda)=\phi(w^{''},\lambda)$.
By the definition of $\sim$, we get \[\alpha_{i}\sim\left\{\begin{array}{ll}
\gamma_{i+1} & \mathrm{if}\ i\leq s \\ 
\gamma_{1} & \mathrm{if}\ i=s+1 \\
\gamma_{i} & \mathrm{if}\ i\geq s+2.
\end{array}\right. \]
Therefore we obtain $\phi(w,\lambda)=\phi(w^{''},\lambda)$ for any $\lambda\in \Lambda$. \\
(3) First we show that $\phi$ is injective. Let $w:x\stackrel{\alpha_{1}}{\rightarrow}x_{1}\stackrel{\alpha_{2}}{\rightarrow}\cdots\stackrel{\alpha_{r}}{\rightarrow}x_{r}$,
$w^{'}:x\stackrel{\beta_{1}}{\rightarrow}x^{'}_{1}\stackrel{\beta_{2}}{\rightarrow}\cdots\stackrel{\beta_{r}}{\rightarrow}x^{'}_{r}$ are paths in $\overrightarrow{L}$. We assume $(\phi(w,\lambda))_{\lambda\in \Lambda}=(\phi(w^{'},\lambda))_{\lambda\in \Lambda}$. Then we show 
$x_{r}=x^{'}_{r}$ with the using of an induction on $r$.

$(r=1)$\; We note that the condition $(c_{3})$ implies $x_{1}\wedge x^{'}_{1}\geq \tau^{-1} x$. Indeed, since $L(x_{1}\wedge x^{'}_{1})$ is a finite distributive lattice and there are arrows $x\rightarrow x_{1}$,\;$x\rightarrow x_{1}^{'}$ in $\overrightarrow{L}(x_{1}\wedge x^{'}_{1})$, we obtain either
$x_{1}=x_{1}\wedge x^{'}_{1}=x^{'}_{1}$ or there are arrows $x_{1}\rightarrow x_{1}\wedge x^{'}_{1},$\;$x^{'}_{1}\rightarrow x_{1}\wedge x^{'}_{1}$ in $\overrightarrow{L}(x_{1}\wedge x^{'}_{1})$\;(i.e.\;in $\overrightarrow{L}$). Therefore if
$x_{1}\neq x^{'}_{1}$, then $\phi(p,\alpha_{1}/\!\raisebox{-2pt}{$\sim$})\geq 2$ for any path $p$
from $x$ to $\tau^{-1}x$. This contradict to the condition $(c_{4})$.

$(r>1)$\;Without loss of generality, we can assume $\alpha_{1}\not\sim \beta_{1}$. Let $s:=\mathrm{min}\{i\mid \alpha_{s}\sim \beta_{1}\}$. Then there is an arrow $x_{i}\stackrel{\alpha^{'}_{i+1}}{\rightarrow} x_{i}\wedge x^{'}_{1}$ for any $i<s$. We note that $\alpha^{'}_{s}\sim \alpha^{'}_{s-1}\sim\cdots \sim\alpha^{'}_{2}\sim \beta_{1}\sim \alpha_{s}$.
Therefore we get $x_{s-1}\wedge x^{'}_{1}=x_{s}$. Now we take a path
\[w^{''}:x\stackrel{\beta_{1}}{\rightarrow} x^{'}_{1}\stackrel{\gamma_{2}}{\rightarrow} x_{1}\wedge x^{'}_{1}\stackrel{\gamma_{3}}{\rightarrow} x_{2}\wedge x^{'}_{1}\stackrel{\gamma_{4}}{\rightarrow}\cdots
\stackrel{\gamma_{s}}{\rightarrow} x_{s}\stackrel{\alpha_{s+1}}{\rightarrow} x_{s+1}\stackrel{\alpha_{s+2}}{\rightarrow}\cdots
\stackrel{\alpha_{r}}{\rightarrow} x_{r},  \]
with $\gamma_{i}\sim \alpha_{i-1}\;(2\leq i<s)$. Since $(\phi(w^{''},\lambda))_{\lambda\in \Lambda}=(\phi(w,\lambda))_{\lambda\in \Lambda}=
(\phi(w^{'},\lambda))_{\lambda\in \Lambda}$, we obtain $x_{r}=x^{'}_{r}$\;(we use a hypothesis of an induction).    
 
 By applying $(2)$ of this Proposition, we obtain that $\phi$ is injective. Now the assertion follows from the definition of $\phi$.\\
(4)\;We only show that $\phi(x)>^{\mathrm{op}} \phi(y)$ implies $x> y$. Let $x,y\in L$ with $\phi(x)>^{\mathrm{op}} \phi(y)$. Suppose $x\not >y$ then there are two paths $x\vee y\stackrel{\alpha_{1}}{\rightarrow}x_{1}\stackrel{\alpha_{2}}{\rightarrow}\cdots \stackrel{\alpha_{s}}{\rightarrow}x_{s}=x$ and
$x\vee y\stackrel{\beta_{1}}{\rightarrow}y_{1}\stackrel{\beta_{2}}{\rightarrow}\cdots \stackrel{\beta_{t}}{\rightarrow}y_{t}=y$. Then $\phi(x)>^{\mathrm{op}}\phi(y)$ implies $I:=\{i\mid \beta_{i}\sim \alpha_{1} \}\neq \emptyset$. We put $i:=\mathrm{min}\;I$. Then
we obtain  following diagram,
\[
\unitlength 0.1in
\begin{picture}( 32.4000,  8.2500)(  3.7000, -9.4500)
\put(5.9000,-4.0000){\makebox(0,0)[lb]{$x\vee y$}}%
\put(12.8000,-4.0000){\makebox(0,0)[lb]{$y_{1}$}}%
\put(18.6000,-4.0000){\makebox(0,0)[lb]{$y_{2}$}}%
%
\special{pn 8}%
\special{pa 900 340}%
\special{pa 1200 340}%
\special{fp}%
\special{sh 1}%
\special{pa 1200 340}%
\special{pa 1134 320}%
\special{pa 1148 340}%
\special{pa 1134 360}%
\special{pa 1200 340}%
\special{fp}%
%
\special{pn 8}%
\special{pa 1520 340}%
\special{pa 1820 340}%
\special{fp}%
\special{sh 1}%
\special{pa 1820 340}%
\special{pa 1754 320}%
\special{pa 1768 340}%
\special{pa 1754 360}%
\special{pa 1820 340}%
\special{fp}%
%
\special{pn 8}%
\special{pa 2040 340}%
\special{pa 2340 340}%
\special{fp}%
\special{sh 1}%
\special{pa 2340 340}%
\special{pa 2274 320}%
\special{pa 2288 340}%
\special{pa 2274 360}%
\special{pa 2340 340}%
\special{fp}%
%
\special{pn 8}%
\special{sh 1}%
\special{ar 2470 340 10 10 0  6.28318530717959E+0000}%
\special{sh 1}%
\special{ar 2540 340 10 10 0  6.28318530717959E+0000}%
\special{sh 1}%
\special{ar 2610 340 10 10 0  6.28318530717959E+0000}%
\put(29.9000,-3.9000){\makebox(0,0)[lb]{$y_{i-1}$}}%
%
\special{pn 8}%
\special{pa 2690 340}%
\special{pa 2690 340}%
\special{fp}%
%
\special{pn 8}%
\special{pa 2670 340}%
\special{pa 2970 340}%
\special{fp}%
\special{sh 1}%
\special{pa 2970 340}%
\special{pa 2904 320}%
\special{pa 2918 340}%
\special{pa 2904 360}%
\special{pa 2970 340}%
\special{fp}%
\put(7.4000,-9.9000){\makebox(0,0)[lb]{$y^{'}_{1}$}}%
\put(12.8000,-10.0000){\makebox(0,0)[lb]{$y^{'}_{2}$}}%
\put(18.6000,-10.0000){\makebox(0,0)[lb]{$y^{'}_{3}$}}%
%
\special{pn 8}%
\special{pa 900 940}%
\special{pa 1200 940}%
\special{fp}%
\special{sh 1}%
\special{pa 1200 940}%
\special{pa 1134 920}%
\special{pa 1148 940}%
\special{pa 1134 960}%
\special{pa 1200 940}%
\special{fp}%
%
\special{pn 8}%
\special{pa 1520 940}%
\special{pa 1820 940}%
\special{fp}%
\special{sh 1}%
\special{pa 1820 940}%
\special{pa 1754 920}%
\special{pa 1768 940}%
\special{pa 1754 960}%
\special{pa 1820 940}%
\special{fp}%
%
\special{pn 8}%
\special{pa 2040 940}%
\special{pa 2340 940}%
\special{fp}%
\special{sh 1}%
\special{pa 2340 940}%
\special{pa 2274 920}%
\special{pa 2288 940}%
\special{pa 2274 960}%
\special{pa 2340 940}%
\special{fp}%
%
\special{pn 8}%
\special{sh 1}%
\special{ar 2470 940 10 10 0  6.28318530717959E+0000}%
\special{sh 1}%
\special{ar 2540 940 10 10 0  6.28318530717959E+0000}%
\special{sh 1}%
\special{ar 2610 940 10 10 0  6.28318530717959E+0000}%
\put(30.7000,-10.0000){\makebox(0,0)[lb]{$y^{'}_{i}$}}%
%
\special{pn 8}%
\special{pa 2670 940}%
\special{pa 2970 940}%
\special{fp}%
\special{sh 1}%
\special{pa 2970 940}%
\special{pa 2904 920}%
\special{pa 2918 940}%
\special{pa 2904 960}%
\special{pa 2970 940}%
\special{fp}%
%
\special{pn 8}%
\special{pa 800 420}%
\special{pa 800 830}%
\special{fp}%
\special{sh 1}%
\special{pa 800 830}%
\special{pa 820 764}%
\special{pa 800 778}%
\special{pa 780 764}%
\special{pa 800 830}%
\special{fp}%
%
\special{pn 8}%
\special{pa 1310 420}%
\special{pa 1310 830}%
\special{fp}%
\special{sh 1}%
\special{pa 1310 830}%
\special{pa 1330 764}%
\special{pa 1310 778}%
\special{pa 1290 764}%
\special{pa 1310 830}%
\special{fp}%
%
\special{pn 8}%
\special{pa 1900 420}%
\special{pa 1900 830}%
\special{fp}%
\special{sh 1}%
\special{pa 1900 830}%
\special{pa 1920 764}%
\special{pa 1900 778}%
\special{pa 1880 764}%
\special{pa 1900 830}%
\special{fp}%
%
\special{pn 8}%
\special{pa 3100 430}%
\special{pa 3100 840}%
\special{fp}%
\special{sh 1}%
\special{pa 3100 840}%
\special{pa 3120 774}%
\special{pa 3100 788}%
\special{pa 3080 774}%
\special{pa 3100 840}%
\special{fp}%
\put(3.7000,-9.8000){\makebox(0,0)[lb]{$x_{1}=$}}%
\put(9.7000,-3.0000){\makebox(0,0)[lb]{$\beta_{1}$}}%
\put(16.0000,-2.9000){\makebox(0,0)[lb]{$\beta_{2}$}}%
\put(5.6000,-6.8000){\makebox(0,0)[lb]{$\alpha_{1}$}}%
\put(31.4000,-6.8000){\makebox(0,0)[lb]{$\gamma$}}%
\put(36.1000,-7.1000){\makebox(0,0)[lb]{$y^{'}_{l}:=y_{l-1}\wedge x_{1}\;(l=1,2\cdots i)$}}%
\end{picture}%
\]
in $\overrightarrow{L}$.
 Since $\gamma\sim \alpha$ we get $x_{1}\wedge y_{i-1}=y_{i}$. Therefore we obtain  $x\vee y> x_{1}\geq y_{i}\geq y$ and $x_{1}\geq x$. We now get a contradiction. \\ 
(5)\; Since $x\vee y \geq x,y$ there are paths,
$x\vee y\stackrel{\alpha_{1}}{\rightarrow}x_{1}\stackrel{\alpha_{2}}{\rightarrow}\cdots \stackrel{\alpha_{s}}{\rightarrow}x_{s}=x$ and $x\vee y\stackrel{\beta_{1}}{\rightarrow}y_{1}\stackrel{\beta_{2}}{\rightarrow}\cdots \stackrel{\beta_{t}}{\rightarrow}y_{t}=y$.

Suppose $I:=\{i\mid \alpha_{i}\sim \beta_{k}\ \mathrm{for\ some}\ k \}\neq \emptyset$. Let $i:=\mathrm{min}\;I$,  $j:=\mathrm{min}\{k\mid \alpha_{i}\sim \beta_{k}\}$ and $\lambda:=\alpha_{i}/\!\raisebox{-2pt}{$\sim$}$. 

 We claim $x_{i}\wedge y_{j-1}= x_{i-1}\wedge y_{j}$.  Since $x_{i-1}\wedge y_{j}\leq x_{i-1}$, $x_{i-1}\wedge y_{j}\leq y_{j}$, we obtain 
\[\phi(x_{i-1}\wedge y_{j})_{\lambda^{'}}\geq \left\{\begin{array}{ll}
\phi(x_{i-1})_{\lambda}+1 & \mathrm{if\ }\lambda^{'}=\lambda \\ 
\phi(x_{i-1})_{\lambda^{'}} & \mathrm{if\ }\lambda^{'}\neq\lambda.
\end{array}\right. \] Above inequalities imply $\phi(x_{i-1}\wedge y_{j})\leq^{\mathrm{op}} \phi(x_{i})$. Therefore we obtain $x_{i-1}\wedge y_{j}\leq x_{i}$. In particular we get $x_{i-1}\wedge y_{j}\leq x_{i}\wedge y_{j-1}$. Similarly we obtain $x_{i}\wedge y_{j-1}\leq x_{i-1}\wedge y_{j}$.    

Since 
\[\begin{array}{lll}
y_{j-1} & = & (x\vee y)\wedge y_{j-1} \\ 
 & = & (x_{i}\vee y_{j})\wedge y_{j-1} \\ 
 & = & (x_{i}\wedge y_{j-1})\vee (y_{j}\wedge y_{j-1}) \\ 
 & = & (x_{i-1}\wedge y_{j})\vee y_{j} \\ 
 & = & y_{j}, \\ 
\end{array} \]
  we get a contradiction. In particular $I=\emptyset$. 
  
Since $\phi(x\vee y)\geq^{\mathrm{op}} (\mathrm{min}\{\phi(x)_{\lambda},\phi(y)_{\lambda}\})_{\lambda}$, it is sufficient to show that $\phi(x\vee y)\not >^{\mathrm{op}} (\mathrm{min}\{\phi(x)_{\lambda},\phi(y)_{\lambda}\})_{\lambda}$. If $\phi(x\vee y) >^{\mathrm{op}} (\mathrm{min}\{\phi(x)_{\lambda},\phi(y)_{\lambda}\})_{\lambda}$, then there exists $\lambda\in \Lambda$ such that $\phi(x\vee y)_{\lambda}<\mathrm{min}\{\phi(x)_{\lambda},\phi(y)_{\lambda}\}$.
This implies there exists $(i,j)$ such that $\alpha_{i}/\!\raisebox{-2pt}{$\sim$}=\lambda=\beta_{j}/\!\raisebox{-2pt}{$\sim$}$. In particular $I\neq \emptyset$. We obtain a contradiction.

\end{proof}

\begin{cor}
\label{c2}
Assume that $(L,\tau^{-1})$ satisfies the conditions $(c_{0})\sim(c_{4})$. Then a map $\phi$ defined in Proposition\;\ref{p1} induces a lattice inclusion \[L\rightarrow (\mathbb{Z}^{\Lambda},\leq^{\mathrm{op}}).\] 
\end{cor}

\begin{lem}
\label{l0}
Assume that $Q$ satisfies the condition $(\mathrm{C})$. Then 
$(\mathcal{T}_{\mathrm{p}}(Q),\tau_{Q}^{-1})$ satisfies the conditions $(c_{0})\sim (c_{4})$. Moreover we get $\Lambda=Q_{0}$ and $\phi(T)=(r_{x})_{x\in Q_{0}}$ for a basic pre-projective tilting module $T\simeq \oplus_{x\in Q_{0}}\tau^{-r_{x}}P(x)$.
\end{lem}

\begin{proof}
For any $T\simeq \oplus_{x\in Q_{0}}\tau^{-r_{x}}P(x)$, we set $T_{x}:=r_{x}$. For any arrow $\alpha:T\rightarrow T^{'}$ in $\ptiltq(Q)$ we set $v(\alpha)\in Q_{0}$ such that $T^{'}_{v(\alpha)}=T_{v(\alpha)}+1$. 

First we will show that $\alpha\sim \beta $ if and only if $v(\alpha)=v(\beta)$. Let
$\alpha:T\rightarrow T^{'}$ and $\beta:T^{''}\rightarrow T^{'''}$. We note that $v(\alpha)=v(\beta)$ if either $\beta=\tau_{Q}^{-1}\alpha $ or there is the diagram,
\[
\unitlength 0.1in
\begin{picture}( 13.8500,  9.4000)( 15.1000,-10.1000)
\put(22.8000,-3.6000){\makebox(0,0)[lb]{$T$}}%
%
\special{pn 8}%
\special{pa 2480 300}%
\special{pa 2750 300}%
\special{fp}%
\special{sh 1}%
\special{pa 2750 300}%
\special{pa 2684 280}%
\special{pa 2698 300}%
\special{pa 2684 320}%
\special{pa 2750 300}%
\special{fp}%
\put(28.3000,-3.6000){\makebox(0,0)[lb]{$T^{'}$}}%
\put(25.7000,-2.4000){\makebox(0,0)[lb]{$\alpha$}}%
%
\special{pn 8}%
\special{pa 2320 420}%
\special{pa 2320 790}%
\special{fp}%
\special{sh 1}%
\special{pa 2320 790}%
\special{pa 2340 724}%
\special{pa 2320 738}%
\special{pa 2300 724}%
\special{pa 2320 790}%
\special{fp}%
%
\special{pn 8}%
\special{pa 2490 920}%
\special{pa 2760 920}%
\special{fp}%
\special{sh 1}%
\special{pa 2760 920}%
\special{pa 2694 900}%
\special{pa 2708 920}%
\special{pa 2694 940}%
\special{pa 2760 920}%
\special{fp}%
\put(22.5000,-9.6000){\makebox(0,0)[lb]{$T^{''}$}}%
\put(28.3000,-9.8000){\makebox(0,0)[lb]{$T^{'''}$}}%
\put(25.9000,-11.8000){\makebox(0,0)[lb]{$\beta$}}%
%
\special{pn 8}%
\special{pa 2890 420}%
\special{pa 2890 790}%
\special{fp}%
\special{sh 1}%
\special{pa 2890 790}%
\special{pa 2910 724}%
\special{pa 2890 738}%
\special{pa 2870 724}%
\special{pa 2890 790}%
\special{fp}%
\put(15.1000,-6.7000){\makebox(0,0)[lb]{$S(\alpha,\beta)$}}%
\end{picture}%
\]\\\\
in $\ptiltq(Q)$. Therefore $\alpha\sim \beta$ implies $v(\alpha)=v(\beta)$. 

We assume $v(\alpha)=v(\beta)=x$. We  show  $\alpha\sim \beta$.    At first we assume $T_{x}=T^{''}_{x}$ and $T\geq T^{''}$. Let $w:T=X^{0}\rightarrow X^{1}\rightarrow\cdots \rightarrow T^{r}=T^{''}$ be a path in $\ptiltq(Q)$. Since $v(X^{i-1}\rightarrow X^{i})\neq x$ for any $i>0$, we obtain a path
$w^{'}:T=Y^{0}\rightarrow Y^{1}\rightarrow Y^{2}\rightarrow \cdots\rightarrow Y^{r}=T^{'''}$, where $Y^{i}:=X^{i-1}\wedge T^{'}\;(1\leq i \leq r)$. Now we note that there is a diagram,
\[
\unitlength 0.1in
\begin{picture}( 18.4000,  9.2000)( 22.5000, -9.9000)
\put(22.6000,-3.6000){\makebox(0,0)[lb]{$X^{i-1}$}}%
%
\special{pn 8}%
\special{pa 2550 310}%
\special{pa 2820 310}%
\special{fp}%
\special{sh 1}%
\special{pa 2820 310}%
\special{pa 2754 290}%
\special{pa 2768 310}%
\special{pa 2754 330}%
\special{pa 2820 310}%
\special{fp}%
\put(29.1000,-3.7000){\makebox(0,0)[lb]{$Y^{i-1}$}}%
\put(25.7000,-2.4000){\makebox(0,0)[lb]{$\alpha_{i-1}$}}%
%
\special{pn 8}%
\special{pa 2320 420}%
\special{pa 2320 790}%
\special{fp}%
\special{sh 1}%
\special{pa 2320 790}%
\special{pa 2340 724}%
\special{pa 2320 738}%
\special{pa 2300 724}%
\special{pa 2320 790}%
\special{fp}%
%
\special{pn 8}%
\special{pa 2550 910}%
\special{pa 2820 910}%
\special{fp}%
\special{sh 1}%
\special{pa 2820 910}%
\special{pa 2754 890}%
\special{pa 2768 910}%
\special{pa 2754 930}%
\special{pa 2820 910}%
\special{fp}%
\put(22.5000,-9.6000){\makebox(0,0)[lb]{$X_{i}$}}%
\put(29.0000,-9.6000){\makebox(0,0)[lb]{$Y^{i}$}}%
\put(25.9000,-11.6000){\makebox(0,0)[lb]{$\alpha_{i}$}}%
%
\special{pn 8}%
\special{pa 2970 420}%
\special{pa 2970 790}%
\special{fp}%
\special{sh 1}%
\special{pa 2970 790}%
\special{pa 2990 724}%
\special{pa 2970 738}%
\special{pa 2950 724}%
\special{pa 2970 790}%
\special{fp}%
\put(40.9000,-6.8000){\makebox(0,0)[lb]{$(\forall i>0)$}}%
\end{picture}%
 \]
In particular we obtain $\alpha\sim \beta$. 

Next we show $\alpha\sim \beta$ in arbitrary case. Since $\alpha\sim \tau_{Q}^{-\mathrm{max}\{0,T^{''}_{x}-T_{x}\}}\alpha$ and $\beta\sim \tau _{Q}^{-\mathrm{max}\{0,T_{x}-T^{''}_{x}\}}\beta$,
  we can assume that $T_{x}=T^{''}_{x}$. Then there is an arrow $\gamma:T\wedge T^{''}\rightarrow T^{'}\wedge T^{'''}$ with $v(\gamma)=x$. Since $(T\wedge T^{''})_{x}=T_{x}=T^{'}_{x}$
 and $T\geq T\wedge T^{''}\leq T^{''}$, we obtain $\alpha\sim \gamma \sim \beta$. 
 
 Therefore $v$ induces $Q_{0}\simeq\Lambda$. In particular, we obtain $\phi(T)_{x}=T_{x}$. Now, by applying Theorem\;\ref{t0}, we can easily check that $(\mathcal{T}_{\mathrm{p}}(Q),\tau_{Q}^{-1})$ satisfies the conditions $(c_{0})\sim (c_{4})$.

\end{proof}

From now on we assume that $(L,\tau^{-1})$ satisfies the conditions $(c_{0})\sim (c_{4})$ in Proposition\;\ref{p1}. Put $\Lambda:=\Lambda(L,\tau^{-1})$. Then we can identify $L$ with its Hasse-quiver $\overrightarrow{L}$. Indeed Proposition\;\ref{p1}\;(1) shows that $x>y$ in $L$ if and only if there exists a
path from $x$ to $y$ in $\overrightarrow{L} $. Moreover, by Proposition\;\ref{p1} and Corollary\;\ref{c2}, we can regard $L$ as a sub-lattice of $(\mathbb{Z}^{\Lambda},\leq^{\mathrm{op}})$. Then for $x\in L$ and $\lambda\in \Lambda$ we
denote by $x_{\lambda}$  the $\lambda$-th entry of $x$. We note that $(\tau^{-1}x)_{\lambda}=x_{\lambda}+1$ for any $x\in L$ and $\lambda\in \Lambda$. Now we define  $\tau x:=(x_{\lambda}-1)_{\lambda\in \Lambda}\in L $ for any $x\in \tau^{-1}L$. 

\begin{lem}
\label{l1}
Let $\lambda\in \Lambda$. Then there is the minimum element $x(\lambda)$
of the set $L(\lambda):=\{x\in L_{0}\mid x_{\lambda}=0\}$. Moreover,
 $\lambda\mapsto x(\lambda)$ induces an inclusion $\Lambda\rightarrow P_{0}$.  

\end{lem}  
\begin{proof}
Since $x_{\lambda}\geq 1$ for any $x\in L_{0}\setminus P_{0}$, we obtain
$\#L(\lambda)<\infty$. Therefore we can take $x(\lambda):=\wedge_{x\in L(\lambda)}x$. 

Now we assume $x(\lambda_{1})=x(\lambda_{2})$. Let $\alpha$ be an arrow
$x(\lambda_{1})\stackrel{\alpha}{\rightarrow} y$ in $L$. Since $y_{\lambda_{1}}\neq 0$ and $y_{\lambda_{2}}\neq 0$, we obtain $\lambda_{1}=\alpha/\!\raisebox{-2pt}{$\sim$}=\lambda_{2}$. 
\end{proof}

\begin{rem}
$\tau^{-r}x(\lambda)$ is the minimum element of $\{x\in L\mid x_{\lambda}\leq r\}$. Indeed for any $x\in L$ with $x_{\lambda}=s\leq r$, we have $x^{'}:=\tau^{-s}o\wedge x\in \tau^{-s}P$ and $\tau^{s}x^{'}_{\lambda}=0$.
Therefore $\tau^{s}x^{'}\geq x(\lambda)$. In particular we have \[x\geq x^{'}\geq \tau^{-s}x(\lambda)\geq \tau^{-r}x(\lambda).\]
\end{rem}

For any $\lambda\in \Lambda$ we denote by $y(\lambda)$ the unique direct successor of $x(\lambda)$.

\begin{deflem}
\label{dl1}
We can define a partial order $\leq$ on $\Lambda$ as follows:
\[\lambda_{1}\leq \lambda_{2}\stackrel{\mathrm{def}}{\Leftrightarrow} x_{\lambda_{1} }\geq x_{\lambda_{2}}\ \forall x\in L.\]

\end{deflem}

\begin{proof}
It is obvious that (1)\;$\lambda\leq \lambda$ for any $\lambda\in \Lambda$ and (2)\;$\lambda_{1}\leq \lambda_{2}\leq \lambda_{3}$ implies $\lambda_{1}\leq \lambda_{3}$. Therefore it is sufficient to show that (3)\;$\lambda_{1}\leq \lambda_{2} \leq \lambda_{1}$ implies $\lambda_{1}=\lambda_{2}$. Since $\lambda_{1}\leq \lambda_{2}$, we get $x(\lambda_{1})_{\lambda_{2}}\leq x(\lambda_{1})_{\lambda_{1}}=0$. Therefore we obtain $x(\lambda_{1})\geq x(\lambda_{2})$. Similarly $\lambda_{2}\leq \lambda_{1}$ implies $x(\lambda_{2})\geq x(\lambda_{1})$. Then Lemma\;\ref{l1} shows $\lambda_{1}=\lambda_{2}$.   

\end{proof}
\begin{lem}
Let $\lambda_{1},\lambda_{2}\in \Lambda$. Then $\lambda_{1}\leq\lambda_{2}$ if and only if $x(\lambda_{1})_{\lambda_{2}}=0$.
\end{lem}

\begin{proof}
First we assume that $\lambda_{1}\leq \lambda_{2}$. Then we get $x(\lambda_{1})_{\lambda_{2}}\leq x(\lambda_{1})_{\lambda_{1}}=0$.

Next we assume that there exists $x\in L$ such that $x_{\lambda_{1}}<x_{\lambda_{2}}$. We consider an element
$c:=\tau^{x_{\lambda_{1}}}(\tau^{-x_{\lambda_{1}}}o\wedge x)\in P$\;(note that $\tau^{-x_{\lambda_{1}}}o\wedge x \in \tau^{-x_{\lambda_{1}}}P$). Since $c_{\lambda_{1}}=0$ and $c_{\lambda_{2}}=x_{\lambda_{2}}-x_{\lambda_{1}}>0$, we obtain $0<c_{\lambda_{2}}\leq x(\lambda_{1})_{\lambda_{2}}$.

\end{proof} 
 
\begin{defn} Let $(L,\tau^{-1})$ be a pair satisfying the conditions $(c_{0})\sim (c_{4})$. We define a quiver
$Q=Q(L,\tau^{-1})$ having $\Lambda$ as the
set of vertices as follows:

We draw an arrow $\lambda\rightarrow \lambda^{'}$ in $Q$ if 
$x(\lambda)_{\lambda^{'}}=1$ and $x(\lambda^{'})_{\lambda}=0$.
\end{defn}

Now we denote by $G(\lambda,\tau^{-1})$ the underlying graph of $Q(\lambda,\tau^{-1})$. Note that there is an edge $\lambda-\lambda^{'}$ in $G(\Lambda,\tau^{-1})$ if and only if $x(\lambda)_{\lambda^{'}}+x(\lambda^{'})_{\lambda}=1$.

\begin{defn}
\label{d1}
Let  $(L,\tau^{-1})$ be a pair satisfying the conditions $(c_{0})\sim (c_{4})$. We define a graph
$G^{'}=G^{'}(L,\tau^{-1})$ having $\Lambda$ as a
set of vertices as follows:

We draw an edge  $\lambda-\lambda^{'}$ in $G^{'}$ if one of the following hold.\\
$(1)$\;there is an edge $\lambda-\lambda^{'}$ in the underlying graph of the Hasse quiver of $(\Lambda,\leq)$.\\
$(2)$\;there is an arrow $\alpha\in L_{1}$ such that $s(\alpha)=y(\lambda)$ and $\alpha/\!\raisebox{-2pt}{$\sim$}=\lambda^{'}$.\\
$(3)$\;there is an arrow $\beta\in L_{1}$ such that $s(\beta)=y(\lambda^{'})$ and $\beta/\!\raisebox{-2pt}{$\sim$}=\lambda$.

\end{defn}

 For any quiver $Q\in\mathcal{Q}$ we define a quiver $\overline{Q}$ satisfying the condition $(\mathrm{C})$ as follows:\\
$(1)$\;$\overline{Q}_{0}=Q_{0}$.\\
$(2)$\;For any pair $(x,\alpha)\in Q_{0}\times Q_{1}$ with $\delta (x)=1$ and $\alpha$ being an edge satisfying either $s(\alpha)=x$ or $t(\alpha)=x$,
 draw new edge $\alpha^{c}:s(\alpha)\rightarrow t(\alpha)$ in $\overline{Q}$. For example, if we consider the following quiver $Q$:
\[
\unitlength 0.1in
\begin{picture}( 11.3000,  0.6500)( 11.7000,-10.3000)
%
\special{pn 8}%
\special{ar 1200 1000 30 30  0.0000000 6.2831853}%
%
\special{pn 8}%
\special{pa 1260 990}%
\special{pa 1660 990}%
\special{fp}%
\special{sh 1}%
\special{pa 1660 990}%
\special{pa 1594 970}%
\special{pa 1608 990}%
\special{pa 1594 1010}%
\special{pa 1660 990}%
\special{fp}%
%
\special{pn 8}%
\special{ar 1720 1000 30 30  0.0000000 6.2831853}%
%
\special{pn 8}%
\special{pa 1780 970}%
\special{pa 2180 970}%
\special{fp}%
\special{sh 1}%
\special{pa 2180 970}%
\special{pa 2114 950}%
\special{pa 2128 970}%
\special{pa 2114 990}%
\special{pa 2180 970}%
\special{fp}%
%
\special{pn 8}%
\special{pa 1780 1020}%
\special{pa 2180 1020}%
\special{fp}%
\special{sh 1}%
\special{pa 2180 1020}%
\special{pa 2114 1000}%
\special{pa 2128 1020}%
\special{pa 2114 1040}%
\special{pa 2180 1020}%
\special{fp}%
%
\special{pn 8}%
\special{ar 2270 1000 30 30  0.0000000 6.2831853}%
\end{picture}%
\]
then $\overline{Q}$ is given by the following:
\[
\unitlength 0.1in
\begin{picture}( 11.3000,  0.6500)( 11.7000,-10.3000)
%
\special{pn 8}%
\special{ar 1200 1000 30 30  0.0000000 6.2831853}%
%
\special{pn 8}%
\special{pa 1260 970}%
\special{pa 1660 970}%
\special{fp}%
\special{sh 1}%
\special{pa 1660 970}%
\special{pa 1594 950}%
\special{pa 1608 970}%
\special{pa 1594 990}%
\special{pa 1660 970}%
\special{fp}%
%
\special{pn 8}%
\special{ar 1720 1000 30 30  0.0000000 6.2831853}%
%
\special{pn 8}%
\special{pa 1780 970}%
\special{pa 2180 970}%
\special{fp}%
\special{sh 1}%
\special{pa 2180 970}%
\special{pa 2114 950}%
\special{pa 2128 970}%
\special{pa 2114 990}%
\special{pa 2180 970}%
\special{fp}%
%
\special{pn 8}%
\special{pa 1780 1020}%
\special{pa 2180 1020}%
\special{fp}%
\special{sh 1}%
\special{pa 2180 1020}%
\special{pa 2114 1000}%
\special{pa 2128 1020}%
\special{pa 2114 1040}%
\special{pa 2180 1020}%
\special{fp}%
%
\special{pn 8}%
\special{ar 2270 1000 30 30  0.0000000 6.2831853}%
%
\special{pn 8}%
\special{pa 1260 1020}%
\special{pa 1660 1020}%
\special{fp}%
\special{sh 1}%
\special{pa 1660 1020}%
\special{pa 1594 1000}%
\special{pa 1608 1020}%
\special{pa 1594 1040}%
\special{pa 1660 1020}%
\special{fp}%
\end{picture}%
\] 

Now we give an  necessary and sufficient condition for $(L,\tau^{-1})$ being isomorphic to $(\mathcal{T}_{\mathrm{p}}(Q),\tau_{Q}^{-1})$\;(i.e.\;there exists poset isomorphism $\rho:L\simeq \mathcal{T}_{\mathrm{p}}(Q)$ such that $\rho(\tau^{-1}x)=\tau^{-1}_{Q}\rho(x)$ holds for any $x\in L$) for some quiver $Q\in \mathcal{Q}$.

\begin{thm}
\label{t2}
Let $L$ be an infinite distributive lattice with the maximum element $o$ and $\tau^{-1}$ be an inner lattice inclusion of $L$ which induces an inner quiver inclusion of $\overrightarrow{L}$.  Then the following are equivalent.\\
$(a)$\;$(L,\tau^{-1})\simeq(\mathcal{T}_{\mathrm{p}}(Q),\tau_{Q}^{-1})$ for some quiver $Q\in\mathcal{Q}$.\\
$(b)$\;$(L,\tau^{-1})$ satisfies the conditions $(c_{0})\sim (c_{4})$ and, \[(c_{5}): G^{'}(L,\tau^{-1})_{1}\subset G(L,\tau^{-1})_{1}.\]

In this case we can take $Q=\overline{Q(L,\tau^{-1})}$. 

 \end{thm}
\begin{proof}
$((a)\Rightarrow (b))$\;Let $Q\in \mathcal{Q}$ such that $(L,\tau^{-1})\simeq (\mathcal{T}_{\mathrm{p}}(Q)),\tau^{-1}_{Q}$. Then Theorem\;\ref{t1} implies that $Q$ satisfies the conditions $(\mathrm{C})$. Therefore Lemma\;\ref{l0} implies that $(\mathcal{T}_{\mathrm{p}}(Q),\tau_{Q}^{-1})$ satisfies the conditions $(c_{0})\sim(c_{4})$. We show that
$(\mathcal{T}_{\mathrm{p}}(Q),\tau_{Q}^{-1})$ satisfies the condition $(c_{5})$.

In this case we note that $\Lambda=Q_{0}$, $x(a)_{b}=l_{Q}(a,b)$ and $a\leq b\Leftrightarrow l_{Q}(a,b)=0$. Let $a,b\in Q_{0}$ satisfying one of the following\;(see Definition\;\ref{d1}):\\
(1)There is an edge $a-b$ in the underlying graph of the Hasse-quiver of $(Q_{0},\leq)$.\\
(2)\;There is an arrow $\alpha\in \ptiltq(Q)_{1}$ such that $s(\alpha)=y(a)$ and $v(\alpha)=b$.\\
(3)\;There is an arrow $\beta\in \ptiltq(Q)_{1}$ such that $s(\beta)=y(b)$ and $\beta/\!\raisebox{-2pt}{$\sim$}=a$.

It is sufficient to show that $l_{Q}(a,b)+l_{Q}(b,a)=1$.
First we assume $(a,b)$ satisfies (1). Then it is obvious that $l_{Q}(a,b)+l_{Q}(b,a)=1$. Next we assume $(a,b)$ satisfies (2). Let $l_{Q}(a,b)=l$. If $l=0$, then there exists a path $a\leftarrow a_{1}\leftarrow \cdots\leftarrow a_{r}=b$ in $Q$. If $a_{1}\neq b$, then 
\[t(\alpha)_{b}=l_{Q}(a,b)+1>0=l_{Q}(a,a_{1})+l_{Q}(a_{1},b)=t(\alpha)_{a_{1}}+l_{Q}(a_{1},b).\] Therefore we get a contradiction. In particular $b=a_{1}$. We assume $l>0$. 
 Then there exists $(a_{1},a_{2}\cdots,a_{l})\in Q_{0}^{l}$ and $(b_{1},\cdots b_{l})\in Q_{0}^{l}$ such that $a\leq a_{1}$, $a_{i}\rightarrow b_{i}\leq a_{i+1}$\;$(i=1,\cdots l)$ and $b_{l}\leq b$ in $Q$. 
If $a\neq a_{1}$, then we obtain \[t(\alpha)_{b}=l_{Q}(a,b)+1>l_{Q}(a,b)=l_{Q}(a,a_{1})+l_{Q}(a_{1},b)=t(\alpha)_{a_{1}}+l_{Q}(a_{1},b).\] Therefore we get a contradiction. If $b\neq b_{1}$, then we obtain \[t(\alpha)_{b}=l_{Q}(a,b)+1>l_{Q}(a,b)=l_{Q}(a,b_{1})+l_{Q}(b_{1},b)=t(\alpha)_{b_{1}}+l_{Q}(b_{1},b).\] We get a contradiction. Therefore $a=a_{1}$ and $b=b_{1}$. In particular we get $l_{Q}(a,b)+l_{Q}(b,a)=1$. Similarly we obtain
$l_{Q}(a,b)+l_{Q}(b,a)=1$ in the case of (3).

$((b)\Rightarrow (a))$\;Let $Q=\overline{Q(L,\tau^{-1})}$. It is sufficient to show that $x\in \mathbb{Z}^{\Lambda}_{\geq 0}$ is in $L$ if and only if
\[(\ast)\cdots x_{\lambda}\leq x_{\lambda^{'}}+l_{Q}(\lambda^{'},\lambda),\ \forall \lambda,\lambda^{'}\in \Lambda.\]

Let $L^{\ast}:=\{z\in \mathbb{Z}_{\geq 0}^{\Lambda}\mid z\ \mathrm{satisfies\ }(\ast)\}(\simeq \mathcal{T}_{\mathrm{p}}(Q))$. For any $x\in L$ we consider $f_{\lambda^{'}}(x):=\tau^{x_{\lambda^{'}}}(\tau^{-x_{\lambda^{'}}}o\wedge x)$.
It is easy to check that, 
\[f_{\lambda^{'}}(x)_{\lambda}=\mathrm{max}\{0,x_{\lambda}-x_{\lambda^{'}}\}.\] Therefore we obtain $f_{\lambda^{'}}(x)\geq x(\lambda^{'})$ and $x_{\lambda}-x_{\lambda^{'}}\leq f_{\lambda^{'}}(x)_{\lambda}\leq x(\lambda^{'})_{\lambda}$. Then we claim, \[(\ast \ast)\cdots x(\lambda)_{\lambda^{'}}\leq l_{Q}(\lambda,\lambda^{'})\ (\forall \lambda,\lambda^{'}\in \Lambda).\]
We note that $(\ast \ast)$ implies $L\subset L^{\ast}$.
Let $l=l_{Q}(\lambda,\lambda^{'})$. Without loss of generality, we can assume $0<l<\infty$. Then there exists $(\lambda_{1},\lambda_{2}\cdots,\lambda_{l})\in \Lambda^{l}$ and $(\lambda^{'}_{1},\cdots \lambda^{'}_{l})\in \Lambda^{l}$ such that $\lambda\leq \lambda_{1}$, $\lambda_{i}\rightarrow \lambda^{'}_{i}\leq \lambda_{i+1}$\;$(i=1,\cdots l)$ in $Q$ and $\lambda^{'}_{l}\leq \lambda^{'}$. In particular we obtain 
\[x(\lambda)_{\lambda^{'}}\leq x(\lambda_{1})_{\lambda^{'}}\leq x(\lambda^{'}_{1})_{\lambda^{'}}+1\leq x(\lambda_{2})_{\lambda^{'}}+1 \leq\cdots \leq x(\lambda^{l})_{\lambda^{'}}+l-1\leq x(\lambda^{'}_{l})_{\lambda^{'}}+l=l.\]

Therefore we obtain $(\ast\ast)$. In particular $L\subset L^{\ast}$. 

We show that $L^{\ast}\setminus L=\emptyset$. Let $z\in L^{\ast}\setminus\{o\}$ and $\lambda^{'}$ be a maximal element of
$\Lambda(z):=\{\lambda\in \Lambda\mid z_{\lambda}\geq z_{\lambda^{''}} \ \forall \lambda^{''}\in \Lambda\}$.
We define $z^{'}\in \mathbb{Z}_{\geq 0}^{\Lambda}$ as follows:
\[z_{\lambda}:=\left\{\begin{array}{ll}
z_{\lambda^{'}}-1 & \mathrm{if\ }\lambda=\lambda^{'} \\ 
z_{\lambda} & \mathrm{if\ }\lambda\neq \lambda^{'}.
\end{array}\right. \] We can easily check that $z^{'}\in L^{\ast}$. Indeed it is sufficient to check
\[z_{\lambda}\leq z_{\lambda^{'}}-1+x(\lambda^{'})_{\lambda},\  \forall\lambda\in \Lambda(z)\setminus \{\lambda^{'}\}).\]
Now above inequalities followed from maximality of $\lambda^{'}$.
In particular there is a path  
\[w:o=z_{0}\rightarrow z_{1}\rightarrow\cdots \rightarrow z_{r}=z,\]
in $\mathbb{Z}_{\geq 0}^{\Lambda}$ such that $z_{i}\in L^{\ast}$. This implies that if 
$L^{\ast}\setminus L\neq \emptyset$, then $\Delta^{\ast}:=\{z\in L^{\ast}\setminus L\mid\ \exists x\in L\ \mathrm{such\ that\ }x\rightarrow z\;\mathrm{in }\ \mathbb{Z}_{\geq 0}^{\Lambda}\}\neq \emptyset$.

 Suppose $L^{\ast}\setminus L\neq \emptyset$ and let
$z\in \Delta^{\ast}$. Then, by the definition of $\Delta^{\ast}$, There exists $(x,\lambda)\in L\times \Lambda$ such that  $x\rightarrow z$ in $\mathbb{Z}_{\geq 0}^{\Lambda}$ and $z_{\lambda}=x_{\lambda}+1$. We consider a path, 
\[x=x^{0}\rightarrow x^{1}\rightarrow \cdots \rightarrow x^{r} =\tau^{-1}x, \] in $L$.
Put $s:=\mathrm{min}\{i\mid x^{i}_{\lambda}=x_{\lambda}+1=z_{\lambda}\}$. Then there is a path,
\[z=z^{0}\rightarrow z^{1}\rightarrow \cdots z^{s-1}=x^{s}\in L,\] in $L^{\ast}$ where $z^{i}:=x^{i}\wedge z$. In particular there exists $t<s$ such that $z^{t-1}\in L^{\ast}\setminus L$ and $z^{t}\in L$. Let $\lambda^{'}\in \Lambda$ such that $x^{t}_{\lambda^{'}}=x^{t-1}_{\lambda^{'}}+1$.  Then we note that 
\[\begin{array}{lllllllll}
x^{t-1}_{\lambda}+1 & = & z^{t-1}_{\lambda} & \leq  z^{t-1}_{\lambda^{'}}+l_{Q}(\lambda^{'},\lambda)&= & x^{t-1}_{\lambda^{'}}+l_{Q}(\lambda^{'},\lambda), \\ 
x^{t-1}_{\lambda^{'}}+1 & = & x^{t}_{\lambda^{'}}  & \leq  x^{t}_{\lambda}+l_{Q}(\lambda,\lambda^{'})&= & x^{t-1}_{\lambda}+l_{Q}(\lambda,\lambda^{'}),
\end{array} \] 
In particular, we obtain $2\leq l_{Q}(\lambda,\lambda^{'})+l_{Q}(\lambda^{'},\lambda)$.

Therefore it is sufficient to prove the following claim $(\dagger)$:\\
$(\dagger)$\ Let $w:x\stackrel{\alpha}{\rightarrow}y\stackrel{\beta}{\rightarrow}z$ be a path in $\overrightarrow{L}$ with $\alpha/\!\raisebox{-2pt}{$\sim$}=\lambda$ and $\beta/\!\raisebox{-2pt}{$\sim$}=\lambda^{'}$. Let $p=(p_{\lambda})_{\lambda\in \Lambda}\in \mathbb{Z}_{\geq 0}^{\Lambda}$ such that $p_{\lambda^{''}}:=\left\{\begin{array}{ll}
x_{\lambda^{'}}+1 &\mathrm{if\ }\lambda^{''}=\lambda^{'} \\ 
x_{\lambda^{''}} & \mathrm{if\ }\lambda^{''}\neq\lambda^{'}.
\end{array}\right. $ If $p\not\in L$, then $l_{Q}(\lambda,\lambda^{'})+l_{Q}(\lambda^{'},\lambda)=1$. 

We prove $(\dagger)$ with the using of an induction on $r:=\mathrm{max}\{z_{\lambda},z_{\lambda^{'}}\}$.\\
$(r=1)$\;We note that $x\geq x(\lambda)$. 
If $x(\lambda)_{\lambda^{'}}>0$, then
\[(x(\lambda)\vee z)_{\lambda^{''}}:=\left\{\begin{array}{ll}
0 &\mathrm{if\ }\lambda^{''}=\lambda \\
1 &\mathrm{if\ }\lambda^{''}=\lambda^{'}\\ 
x_{\lambda^{''}} & \mathrm{otherwise }.
\end{array}\right.\] 
In particular we get $p=x(\lambda)\vee z\in L$. Therefore we get a contradiction. Since $x(\lambda)_{\lambda^{'}}=0$, we obtain $\lambda<\lambda^{'}$.
If there exists $\lambda^{''}\in \Lambda$ such that $\lambda<\lambda^{''}<\lambda^{'}$, then $z_{\lambda^{''}}=x_{\lambda^{''}}\leq x_{\lambda}=0<1=z_{\lambda^{'}}$. This is a contradiction. Therefore the condition $(c_{5})$ implies that there is an arrow $\lambda^{'}\rightarrow \lambda$ in $Q$. In particular we get $l_{Q}(\lambda,\lambda^{'})+l_{Q}(\lambda^{'},\lambda)=1$.\\
$(r>1)$\;We consider the following three cases:\\
(1)\;$x_{\lambda}>0$ and $x_{\lambda^{'}}>0$.
(2)\;$x_{\lambda}=0$ and $x_{\lambda^{'}}>0$.
(3)\;$x_{\lambda}>0$ and $x_{\lambda^{'}}=0$.

In the case of $(1)$ let $x^{'}:=\tau(\tau^{-1}o\wedge x)$, $y^{'}:=\tau(\tau^{-1}o\wedge y)$, and $z^{'}:=\tau(\tau^{-1}\wedge z)$. Then there exists a path $x^{'}\stackrel{\alpha^{'}}{\rightarrow} y^{'}\stackrel{\beta^{'}}{\rightarrow} z^{'}$. Let $p^{'}\in \mathbb{Z}_{\geq 0}^{\Lambda} $ with $p^{'}_{\lambda^{''}}:=\left\{\begin{array}{ll}
x^{'}_{\lambda^{'}}+1 &\mathrm{if\ }\lambda^{''}=\lambda^{'} \\ 
x^{'}_{\lambda^{''}} & \mathrm{if\ }\lambda^{''}\neq\lambda^{'}.
\end{array}\right.$ We note that $\alpha^{'}/\!\raisebox{-2pt}{$\sim$}=\lambda$ and $\beta^{'}/\!\raisebox{-2pt}{$\sim$}=\lambda^{'}$. If $p^{'}\in L$, then $p=(\tau^{-1}p^{'}\vee z)\wedge x \in L$. This is a contradiction.
Therefore we obtain $p^{'}\not\in L$. Since $z^{'}_{\lambda}=z_{\lambda}-1$ and
$z^{'}_{\lambda^{'}}=z_{\lambda^{'}}-1$, the assertion follows from the hypothesis of induction.

In the case of (2), we first show that $x_{\lambda^{'}}=x(\lambda)_{\lambda^{'}}$. If $x(\lambda)_{\lambda^{'}}>x_{\lambda^{'}}$, then we get $p=x(\lambda)\vee z\in L$. This is a contradiction. Therefore, since $x\geq x(\lambda)$, we obtain $x_{\lambda^{'}}=x(\lambda)_{\lambda^{'}}$.  Now let $z^{'}:=y(\lambda)\wedge z$. Since 
\[z^{'}_{\lambda^{''}}=\left\{ \begin{array}{ll}
1 &\mathrm{if\ }\lambda^{''}=\lambda \\ 
x(\lambda)_{\lambda^{'}}+1 & \mathrm{if\ }\lambda^{''}=\lambda^{'}\\
x(\lambda)_{\lambda^{''}} & \mathrm{otherwise},
\end{array}\right.  \]   we obtain a path $x(\lambda)\rightarrow y(\lambda)\stackrel{\beta^{'}}{\rightarrow} z^{'}$ with $\beta^{'}/\!\raisebox{-2pt}{$\sim$}=\lambda^{'}$. Therefore the assertion follows from the condition $(c_{5})$.

Finally we consider the case of (3). Let $p{'}:=(\tau^{-1}o\vee z)\wedge x$. Then it is easy to check that $p_{\lambda^{''}}=p^{'}_{\lambda^{''}}$ for any $\lambda^{''}\in \Lambda$. In particular we get $p\in L$. This is a contradiction.

\end{proof}

\begin{cor}
An infinite distributive lattice $L$ is isomorphic to $\mathcal{T}_{\mathrm{p}}(Q)$ for some $Q\in \mathcal{Q}$ 
if and only if there is a poset inclusion $\tau^{-1}:L\rightarrow L$ which induces a quiver inclusion $\overrightarrow{L}\rightarrow \overrightarrow{L} $ and satisfies the conditions $(c_{0})\sim(c_{5})$. 
\end{cor}

\end{document}